\documentclass[a4paper,12pt]{amsart}
\usepackage{amsmath,amsthm,amssymb,amscd,mathrsfs,color,setspace}
\usepackage[all]{xy}
\usepackage{enumerate}
\usepackage[dvips]{graphicx}
\usepackage{latexsym}
\numberwithin{equation}{section}
\newtheorem{thm}{Theorem}[section]
\newtheorem{prop}[thm]{Proposition}
\newtheorem{cor}[thm]{Corollary}
\newtheorem{lem}[thm]{Lemma}
\theoremstyle{definition}
\newtheorem{df}[thm]{Definition}
\newtheorem{ex}[thm]{Example}
\newtheorem{rmk}[thm]{Remark}
\newtheorem{ques}[thm]{Question}
\def\gldim{\mathop{\mathrm{gl.dim}}\nolimits}
\def\IM{\mathop{\mathrm{Im}}\nolimits}

\def\Soc{\mathop{\mathrm{Soc}}\nolimits}

\def\add{\mathop{\mathrm{add}}\nolimits}
\def\thick{\mathop{\mathrm{thick}}\nolimits}

\def\mod{\mathop{\mathrm{mod}}\nolimits}
\def\Mod{\mathop{\mathrm{Mod}}\nolimits}
\def\proj{\mathop{\mathrm{proj}}\nolimits}

\newcommand{\Hom}{\mathrm{Hom}}
\newcommand{\End}{\mathrm{End}}
\newcommand{\Ext}{\mathrm{Ext}}

\newcommand{\Lten}{\stackrel{\mathbb L}{\otimes}}

\newcommand{\op}{\mathrm{op}}

\newcommand{\dg}{\mathrm{dg}}

\title{Realizing stable categories as derived categories}
\author{Kota Yamaura}
\address{Graduate School of Mathematics\\ 
 Nagoya University, Chikusa-ku Nagoya 464-8602 Japan}
\email{m07052d@math.nagoya-u.ac.jp}
\thanks{The author is supported by JSPS Fellowships for Young Scientists No.22-5801.
}

\begin{document}

\maketitle

\begin{abstract}
In this paper, we discuss a relationship between representation theory of graded self-injective algebras and that of algebras of finite global dimension.  
For a positively graded self-injective algebra $A$ such that $A_0$ has finite global dimension, we construct two types of triangle-equivalences.
First we show that there exists a triangle-equivalence between the stable category of $\mathbb{Z}$-graded $A$-modules and the derived category of a certain algebra $\Gamma$ of finite global dimension. 
Secondly we show that if $A$ has Gorenstein parameter $\ell$, then there exists a triangle-equivalence between the stable category of $\mathbb{Z}/\ell\mathbb{Z}$-graded $A$-modules and a derived-orbit category of $\Gamma$, which is a triangulated hull of the orbit category of the derived category.
\end{abstract}

\tableofcontents

\section{Introduction}
\label{intro_section}

The notion of triangulated categories was introduced by A. Grothendieck and J. L. Verdier in their work of derived categories. 
It appears in several parts of mathematics, for example, representation theory, algebraic geometry and algebraic topology.
In representation theory of algebras, there are many important triangulated categories. 
In this paper, we study two major classes of triangulated categories, i.e. derived categories and stable categories.

The derived category of an abelian category is a powerful tool which simplifies the study of homological properties of the abelian category. 
In representation theory of algebras, the notion of derived equivalences is important since it preserves a lot of homological properties of algebras.
Moreover the derived category connects representation theory of algebras with various parts of mathematics, for example, algebraic geometry and mathematical physics.

In 1980s D. Happel introduced a different kind of triangulated categories in \cite{Ha1}.
He showed that the stable category of a Frobenius category has a natural structure of a triangulated category. 
These triangulated categories are called algebraic. 
A typical example is given by self-injective algebras. 
The category $\mod A$ of modules over a self-injective algebra $A$ is a Frobenius category, so the stable category $\underline{\mod}A$ is an algebraic triangulated category.
Therefore we can study $\mod A$ by applying general theory of algebraic triangulated categories.

There is an important relationship between the stable categories of self-injective algebras and the derived categories of algebras. 
Happel studied their relationship from the following viewpoint. 
For any algebra $\Lambda$ over a field $K$, one can associate the trivial extension $A=\Lambda \oplus D\Lambda$ where $D$ is $K$-dual. 
This is a positively graded self-injective algebra with $A_0=\Lambda$ and $A_1=D\Lambda$.
He showed that if $\Lambda$ has finite global dimension, then there exists a triangle-equivalence 
\begin{eqnarray}
\underline{\mod}^{\mathbb{Z}}A \simeq \mathsf{D}^{\mathrm{b}}(\mod\Lambda) \label{intro_equation}
\end{eqnarray}
where $\underline{\mod}^{\mathbb{Z}}A$ is the stable category of the category of $\mathbb{Z}$-graded $A$-modules, and $\mathsf{D}^{\mathrm{b}}(\mod\Lambda)$ is the bounded derived category of $\mod\Lambda$.
This equivalence played an important role in representation theory. 
Although these algebras $\Lambda$ and $A$ are quite different from homological viewpoint, their representation theories are closely related.

The first aim of this paper is to generalize Happel's equivalence.
From Happel's result it is natural to ask the following question.

\begin{ques}\label{intro_Q1}
Let $A$ be a $\mathbb{Z}$-graded self-injective algebra.
When is the stable category $\underline{\mod}^{\mathbb{Z}}A$ triangle-equivalent to the derived category $\mathsf{D}^{\mathrm{b}}(\mod\Lambda)$ for some algebra $\Lambda$ ?
\end{ques}

Our approach to this question is to apply the notion of tilting objects (Definition \ref{df_tilting}) and tilting theorem (Theorem \ref{Keller}). 
If a Krull-Schmidt algebraic triangulated category $\mathcal{T}$ has a tilting object $T$, then $\mathcal{T}$ is triangle-equivalent to  the homotopy category $\mathsf{K}^{\mathrm{b}}(\proj \End_{\mathcal{T}}(T))$ of bounded complexes of finitely generated projective $\End_{\mathcal{T}}(T)$-modules.

Moreover it is well-known that if an algebra $\Lambda$ has finite global dimension, then the canonical embedding $\mathsf{K}^{\mathrm{b}}(\proj \Lambda) \rightarrow \mathsf{D}^{\mathrm{b}}(\mod \Lambda)$ is an equivalence.
These observations tell us that it is reasonable to divide Question \ref{intro_Q1} into the following questions.

\begin{ques}\label{intro_Q2}
Let  $A$ be a $\mathbb{Z}$-graded self-injective algebra.
\begin{enumerate}
\def\labelenumi{(\theenumi)}
\item When does the stable category $\underline{\mod}^{\mathbb{Z}}A$ have tilting objects ?
\item If $\underline{\mod}^{\mathbb{Z}}A$ has tilting objects, do their endomorphism algebras have finite global dimension ?
\end{enumerate}
\end{ques}

The following first main result in this paper gives a complete answer to Questions \ref{intro_Q1} and \ref{intro_Q2} for the positively graded case.

\begin{thm}[Theorem \ref{main_thm2}, \ref{main_thm3}]\label{main_thm} 
Let $A$ be a positively graded self-injective algebra.
The following conditions are equivalent.
\begin{enumerate}
\def\labelenumi{(\theenumi)} 
\item $A_0$ has finite global dimension.
\item $\underline{\mod}^{\mathbb{Z}}A$ has tilting objects.
\item There exists a triangle-equivalence $\underline{\mathrm{mod}}^{\mathbb{Z}}A \simeq \mathsf{K}^{\mathrm{b}}(\mathrm{proj}\Lambda)$ for some algebra $\Lambda$.
\item There exists a triangle-equivalence $\underline{\mathrm{mod}}^{\mathbb{Z}}A \simeq \mathsf{D}^{\mathrm{b}}(\mathrm{mod}\Lambda)$ for some algebra $\Lambda$.
\end{enumerate}
If $A$ satisfies these equivalent conditions, then any algebra $\Lambda$ in (3) has finite global dimension.
\end{thm}

Applying Theorem \ref{main_thm} to the trivial extension $A$ of $\Lambda$, we recover Happel's equivalence \eqref{intro_equation}.

\medskip

The second aim of this paper is to study the stable categories $\underline{\mod}^{\mathbb{Z}/a\mathbb{Z}}A$ for positively graded self-injective algebras $A$ and positive integers $a$.
The category $\underline{\mod}^{\mathbb{Z}/a\mathbb{Z}}A$, especially $\underline{\mod}A$ for the case $a=1$, is a classical object in representation theory. 
The category $\underline{\mod}^{\mathbb{Z}}A$ has the grade shift functor $(a):\underline{\mod}^{\mathbb{Z}}A \rightarrow \underline{\mod}^{\mathbb{Z}}A$. 
We have a fully faithful functor $(\underline{\mod}^{\mathbb{Z}}A)/(a) \rightarrow \underline{\mod}^{\mathbb{Z}/a\mathbb{Z}}A$ where $(\underline{\mod}^{\mathbb{Z}}A)/(a)$ is the orbit category. 
This is not an equivalence in general, and their difference has been widely studied.
At least the image of this functor generates $\underline{\mod}^{\mathbb{Z}/a\mathbb{Z}}A$ as a triangulated category.
These observations suggest that $\underline{\mod}^{\mathbb{Z}/a\mathbb{Z}}A$  will be triangle-equivalent to some derived-orbit category defined as follows.

Let $\Lambda$ be an algebra of finite global dimension, and $M$ be a bounded complex of $\Lambda^{\op} \otimes_K \Lambda$-modules 
such that $F:=- \Lten_{\Lambda} M: \mathsf{D}^{\mathrm{b}}(\mod\Lambda) \rightarrow \mathsf{D}^{\mathrm{b}}(\mod\Lambda)$ is a triangle-equivalence. 
Although the orbit category $\mathsf{D}^{\mathrm{b}}(\mod\Lambda)/F$ does not necessarily have a structure of a triangulated category, 
we can construct a triangulated category $\mathsf{D}(\Lambda,M)$ which contains $\mathsf{D}^{\mathrm{b}}(\mod\Lambda)/F$ as a full subcategory and is generated by $\mathsf{D}^{\mathrm{b}}(\mod\Lambda)/F$ (see \cite{Ke2}).
We call $\mathsf{D}(\Lambda,M)$ the derived-orbit category. 

The first example of derived-orbit categories is cluster categories $\mathsf{D}(\Lambda,D\Lambda[-2])$.
It was introduced in \cite{Ami1,Ami2,BMRRT,Ke2} from the viewpoint of cluster tilting theory and their applications to the study of cluster algebras.

Now we pose the following question.

\begin{ques}\label{intro_Q3}
Let $A$ be a positively graded self-injective algebra such that $A_0$ has finite global dimension and $a$ a positive integer.
Is the stable category $\underline{\mod}^{\mathbb{Z}/a\mathbb{Z}}A$ triangle-equivalent to some derived-orbit category $\mathsf{D}(\Lambda,M)$ ?
\end{ques}

We study this question for a positively graded self-injective algebra satisfying the following condition.

\begin{df}\label{df_Gor_par}
Let $A$ be a positively graded self-injective algebra. 
We say that $A$ has \emph{Gorenstein parameter} $\ell$ if $\Soc A$ is contained in $A_{\ell}$.
\end{df}

The following second main theorem in this paper gives a partial answer to Question \ref{intro_Q3}.

\begin{thm}[Theorem \ref{hull_vs_graded}]\label{intro_thm3}
Let $A$ be a positively graded self-injective algebra of Gorenstein parameter $\ell$. 
Assume that $A_0$ has finite global dimension.
Then there exist an algebra $\Gamma$ of finite global dimension, a $\Gamma^{\op} \otimes_K \Gamma$-module $M$, and  a triangle-equivalence 
\[
\underline{\mod}^{\mathbb{Z}/\ell\mathbb{Z}}A \simeq \mathsf{D}(\Gamma,M[1]).
\]
\end{thm}

Thanks to our assumption on Gorenstein parameter, we can give an algebra $\Gamma$ and a $\Gamma^{\op} \otimes_K \Gamma$-module $M$ quite explicitly.

As a special case of Theorem \ref{intro_thm3}, we have the following result.

\begin{cor}\label{intro_cor}
Let $\Lambda$ be an algebra, and $A$ be the trivial extension of $\Lambda$.
If $\Lambda$ has finite global dimension, then there exists a triangle-equivalence 
\[
\underline{\mod}A \simeq \mathsf{D}(\Lambda,D\Lambda[1]).
\]
\end{cor}

This means that the stable category of the trivial extension of $\Lambda$ is triangle-equivalent to the ``$(-1)$-cluster category'' $\mathsf{D}(\Lambda,D\Lambda[1])$.

\bigskip

This paper is organized as follows. 
In section \ref{pre_section}, we recall basic facts on triangulated categories and representation theory of graded algebras which are used in the latter sections.

Our first main Theorem \ref{main_thm} is given in sections \ref{cal_section}.
In subsection \ref{existence_tilt_section}, we give a proof of Theorem \ref{main_thm} (1) $\Leftrightarrow$ (2) $\Leftrightarrow$ (3).
We show implication (1) $\Rightarrow$ (2) by constructing a tilting object $T$ in $\underline{\mod}^{\mathbb{Z}}A$ explicitly.
In subsections \ref{calculation_section} and \ref{global_section}, we give some properties of the endomorphism algebra $\Gamma$ of $T$. 
By applying them,  we give a proof of Theorem \ref{main_thm} (1) $\Leftrightarrow$ (4) in subsections \ref{tri_equ_section} and \ref{direct_section}. 
Then we give examples of Theorem \ref{main_thm}.
In section \ref{app_pp_section}, we give a new proof of a result \cite[Theorem 4.7]{IO2} on higher preprojective algebra as an application of Theorem \ref{main_thm}.
Moreover we discuss the case of classical preprojective algebras.

In section \ref{DG_section}, we explain basic definitions and facts about DG categories.
In particular we define derived-orbit categories which play a central role in section \ref{Gor_section}. 
Moreover we give a universal property of derived-orbit categories which plays an important role in the proof of Theorem \ref{intro_thm3}.

Our second main Theorem \ref{intro_thm3} is given in section \ref{Gor_section}.
For a positively graded self-injective algebra $A$ of Gorenstein parameter $\ell$, we give an algebra $\Gamma$ and a $\Gamma^{\op}\otimes_K\Gamma$-module $M$ in Theorem \ref{intro_thm3} explicitly.
Then we construct a triangle-equivalence $\mathsf{D}(\Gamma,M[1]) \rightarrow \underline{\mod}^{\mathbb{Z}/\ell\mathbb{Z}}A$ by applying a universal property of derived-orbit categories.
Finally we give examples of Theorem \ref{intro_thm3}.

\bigskip

\noindent \textbf{Conventions.} Throughout this paper, let $K$ be an algebraically closed field. 
An algebra means a finite dimensional associative algebra over $K$.
For an algebra $\Lambda$, $J_{\Lambda}$ means the Jacobson radical of $\Lambda$.
We always deal with finitely generated right modules. 
We denote by $\mod \Lambda$ the category of $\Lambda$-modules, and by $\proj \Lambda$ the category of projective $\Lambda$-modules. 
We denote the $K$-dual by $D:=\Hom_K(-,K)$.

For a functor $F:\mathcal{A} \rightarrow \mathcal{B}$ between categories $\mathcal{A}$ and $\mathcal{B}$, we denote by $\mathrm{Im}F$ the full subcategory of $\mathcal{B}$ which consists of $F(X)$ for $X \in \mathcal{A}$.

\section{Preliminaries}
\label{pre_section}

In this section, we recall basic facts on triangulated categories and representation theory of group graded self-injective algebras. 
In particular we recall tilting theorem for algebraic triangulated categories which plays an important role in the next section.

\subsection{Derived categories}
In this subsection, we recall basic facts about derived categories.
For background for triangulated categories or derived categories, we refer to \cite{Ha1,Har,Ne}.

For an algebra $\Lambda$, we denote by $\mathsf{K}^{\mathrm{b}}(\proj\Lambda)$ the homotopy category of bounded complexes over $\proj\Lambda$, 
and by $\mathsf{D}^{\mathrm{b}}(\mod\Lambda)$ the bounded derived category of $\mod\Lambda$.
They have structure of triangulated categories and the canonical embedding $\mathsf{K}^{\mathrm{b}}(\proj\Lambda) \rightarrow \mathsf{D}^{\mathrm{b}}(\mod\Lambda)$ is a triangle-functor.

They have properties defined as follows.

\begin{df}\label{df_Hom-finite} \ 
\begin{enumerate}
\def\labelenumi{(\theenumi)} 
\item A $K$-linear category $\mathcal{A}$ is called \emph{Hom-finite} if its morphism spaces are finite dimensional over $K$.
\item An additive category $\mathcal{A}$ is called \emph{Krull-Schmidt} if any object of $\mathcal{A}$ is decomposed into a finite direct sum of objects in $\mathcal{A}$ which have local endomorphism rings. 
\end{enumerate}
\end{df}

\begin{prop}\label{derived_KS}
For an algebra $\Lambda$, the homotopy category $\mathsf{K}^{\mathrm{b}}(\proj\Lambda)$ and the derived category $\mathsf{D}^{\mathrm{b}}(\mod\Lambda)$ are Hom-finite Krull-Schmidt categories.
\end{prop}

We have equivalent conditions for finiteness of global dimension below.
We refer to \cite{RVdB} for the details of Serre functors, and refer to \cite[Chapter I.4]{Ha1} for the details of Auslander-Reiten triangles.


\begin{thm} \cite[Corollary 1.5]{Ha2} \cite[Theorem I.2.4]{RVdB} \label{Serre=gldim}
Let $\Lambda$ be an algebra. The following conditions are equivalent. 
\begin{enumerate}
\def\labelenumi{(\theenumi)} 
\item $\Lambda$ has finite global dimension.
\item The canonical embedding $\mathsf{K}^{\mathrm{b}}(\proj \Lambda) \rightarrow \mathsf{D}^{\mathrm{b}}(\mod\Lambda)$ is an equivalence.
\item $\mathsf{D}^{\mathrm{b}}(\mod\Lambda)$ has a Serre functor.
\item $\mathsf{D}^{\mathrm{b}}(\mod\Lambda)$ has Auslander-Reiten triangles.
\end{enumerate}
If these conditions are satisfied, a Serre functor is given by $-\Lten_{\Lambda}D\Lambda$, and the Auslander-Reiten translation is given by $-\Lten_{\Lambda}D\Lambda[-1]$.
\end{thm}

\subsection{Algebraic triangulated categories}
\label{tilting_section}

In this subsection,  we recall tilting theorem for algebraic triangulated categories which is due to Keller \cite{Ke1}. 
It provides us a method to compare given triangulated categories with derived categories of rings.

First let us recall the definitions of Frobenius categories and its stable categories.
We refer to \cite{Ke3,Qu} for more axiomatic definition of Frobenius categories.

\begin{df}\cite{He} \cite[Chapter I. 2]{Ha1} \label{df_Frobenius}
Let $\mathcal{A}$ be an abelian category. 
\begin{enumerate}
\def\labelenumi{(\theenumi)} 
\item A full subcategory $\mathcal{B}$ of $\mathcal{A}$ is called \emph{extension closed} if for any short exact sequence $0 \rightarrow X \rightarrow Y \rightarrow Z \rightarrow 0$, the condition $X, Z \in \mathcal{B}$ implies $Y \in \mathcal{B}$.
\item Let $\mathcal{B}$ be an extension closed full subcategory of $\mathcal{A}$. 
An object $X$ in $\mathcal{B}$ is called \emph{relative-projective} (respectively, \emph{relative-injective}) if $\Ext^1_{\mathcal{A}}(X,\mathcal{B})=0$ (respectively, $\Ext^1_{\mathcal{A}}(\mathcal{B},X)=0$) holds.
\item An extension closed full subcategory $\mathcal{B}$ of $\mathcal{A}$ is called \emph{Frobenius} if it satisfies the following conditions. 
\begin{enumerate}
\def\labelenumi{(\theenumi)} 
\item For any object $X \in \mathcal{B}$, there exist exact sequences $0 \rightarrow Y \rightarrow P \rightarrow X \rightarrow 0$ and $0 \rightarrow X \rightarrow I \rightarrow Z \rightarrow 0$ in $\mathcal{A}$ such that each term is in $\mathcal{B}$, $P$ is relative-projective, and $I$ is relative-injective.
\item Relative-projective objects coincide with relative-injective objects.
\end{enumerate}
\end{enumerate}
\end{df}

\begin{df}
For a Frobenius category $\mathcal{B}$, the \emph{stable category} $\underline{\mathcal{B}}$ of $\mathcal{B}$ is defined as follows. 
\begin{list}{$\bullet$}{}
\item The objects are the same as $\mathcal{B}$.
\item For $X,Y \in \mathcal{B}$, the morphism set from $X$ to $Y$ is defined by the factor 
\[
\Hom_{\underline{\mathcal{B}}}(X,Y)=\underline{\Hom}_{\mathcal{B}}(X,Y):=\Hom_{\mathcal{B}}(X,Y)/\mathcal{P}(X,Y)
\]
where $\mathcal{P}(X,Y)$ is the subset of $\Hom_{\mathcal{B}}(X,Y)$ which consists of morphisms factoring through projective objects.
\end{list}
\end{df}

The stable category of a Frobenius category has a natural structure of a triangulated category (see \cite[Chapter I. 2]{Ha1}).

\begin{df}\cite{Ha1,Ke3,Kra}\label{df_algebraic}
A triangulated category $\mathcal{T}$ is \emph{algebraic} if it is triangle-equivalent to the stable category of some Frobenius category.
\end{df}

The class of algebraic triangulated categories contains the following important examples (see also Lemma \ref{Gra-Fro}).

\begin{ex} 
Let $\Lambda$ be an algebra. 
The category $\mathsf{C}^{\mathrm{b}}(\proj \Lambda)$ of bounded complexes of projective $\Lambda$-modules has a structure of a Frobenius category. 
The stable category of $\mathsf{C}^{\mathrm{b}}(\proj \Lambda)$ coincides with $\mathsf{K}^{\mathrm{b}}(\proj \Lambda)$. 
So $\mathsf{K}^{\mathrm{b}}(\proj \Lambda)$ is an algebraic triangulated category (cf. \cite[Chapter I, 3 .2]{Ha1}).
\end{ex}

In tilting theory, an important role is played by tilting objects defined as below.

\begin{df}\label{df_thick}
Let $\mathcal{T}$ be a triangulated category. 
\begin{enumerate}
\def\labelenumi{(\theenumi)} 
\item A \emph{thick subcategory} of $\mathcal{T}$ is a triangulated full subcategory of $\mathcal{T}$ which is closed under direct summands and isomorphisms. 
\item For a full subcategory $\mathcal{X}$ (respectively, object $X$) of $\mathcal{T}$, we define $\thick_{\mathcal{T}} \mathcal{X}$ (respectively, $\thick_{\mathcal{T}}X$) as the smallest thick subcategory of $\mathcal{T}$ which contains $\mathcal{X}$ (respectively, $X$).
We say that a full subcategory $\mathcal{X}$ of $\mathcal{T}$ \emph{generates} $\mathcal{T}$ \emph{as a triangulated category} if $\thick \mathcal{X}=\mathcal{T}$ holds.
\end{enumerate}
\end{df}

\begin{df}\cite{Ke1,Ric1}\label{df_tilting}
Let $\mathcal{T}$ be a triangulated category. 
An object $T \in \mathcal{T}$ is called \emph{tilting} if it satisfies the following conditions.
\begin{enumerate}
\def\labelenumi{(\theenumi)} 
\item $\Hom_{\mathcal{T}}(T,T[i])=0$ for any $i \neq 0$.
\item $\mathcal{T}=\thick_{\mathcal{T}} T$.
\end{enumerate}
\end{df}

The following is a typical example of tilting objects.

\begin{ex}\label{ex_tilt}
Let $\Lambda$ be a ring. 
We regard $\Lambda$ as a complex concentrated in degree $0$.
Then $\Lambda$ is a tilting object in $\mathsf{K}^{\mathrm{b}}(\proj \Lambda)$.
\end{ex}

Now we state Keller's tilting theorem.
We refer to  \cite[Appendix]{IT} for the following form.

\begin{thm}\label{Keller} \cite{BK} \cite[Theorem 4.3]{Ke1}
Let $\mathcal{T}$ be an algebraic triangulated Krull-Schmidt category.
If $\mathcal{T}$ has a tilting object $T$, then there exists a triangle-equivalence
\[
\mathcal{T} \simeq \mathsf{K}^{\mathrm{b}}(\proj\End_{\mathcal{T}}(T)).
\]
\end{thm}

Thus it is a basic problem in the study of algebraic triangulated categories to find tilting objects.
We study this for the stable categories of self-injective algebras in the subsection \ref{existence_tilt_section}.

In the last of this subsection, we give the following easy observations for general triangulated categories.

\begin{lem}\label{uses_five_Hom}
Let $\mathcal{T}$ be a $K$-linear triangulated category, and  $\mathcal{X}$ be a full subcategory of $\mathcal{T}$ which is closed under shifts $[\pm 1]$ and satisfies $\thick_{\mathcal{T}}{\mathcal{X}}=\mathcal{T}$. 
If $\mathcal{X}$ is Hom-finite, then so is $\mathcal{T}$.
\end{lem}

\begin{lem}\label{uses_five}
Let $\mathcal{T}$, $\mathcal{U}$ be triangulated categories, and $\mathcal{X}$ be a full subcategory of $\mathcal{T}$ which is closed under shifts $[\pm 1]$.
For a triangle-functor $F:\mathcal{T} \rightarrow \mathcal{U}$, the following assertions hold.
\begin{enumerate}
\def\labelenumi{(\theenumi)}
\item If $\thick\mathcal{X}=\mathcal{T}$ and the restriction functor $F |_{\mathcal{X}}$ is fully faithful, then so is $F$.
\item If $\mathcal{T}$ is Krull-Schmidt, $\thick(\mathrm{Im}F)=\mathcal{U}$ and $F$ is full, then $F$ is dense.
\end{enumerate}
\end{lem}

\subsection{Group graded algebras}

In this subsection, we recall basic facts on representation theory for group graded algebras.
We refer to \cite{GG1,GG2} for details in the case $G=\mathbb{Z}$.

We start with setting notations.
Let $G$ be an abelian group, and $A=\bigoplus_{i \in G}A_i$ be a $G$-graded algebra.
For a $G$-graded $A$-module $X$, we write $X_i$ the degree $i$ part of $X$.

In the following, we give basic properties of the category $\mathrm{mod}^{G}A$ of $G$-graded $A$-modules, which is defined as follows.

\begin{list}{$\bullet$}{}
\item The objects are $G$-graded $A$-modules.
\item For $G$-graded $A$-modules $X$ and $Y$, the morphism space from $X$ to $Y$ in $\mod^GA$ is defined by
\[
\Hom_A(X,Y)_0:=\{f \in \Hom_A(X,Y) \ | \ \mbox{$f(X_i) \subset Y_i$ for any $i \in G$} \}.
\]
\end{list}
We denote by $\proj^{G}A$ the full subcategory of $\mod^{G}A$ consisting of projective objects.

We recall that $\mathrm{mod}^{G}A$ has two important functors.
The first one is the grade shift functor.
For $i \in G$, we denote by 
\[
(i): \mathrm{mod}^{G}A \rightarrow \mathrm{mod}^{G}A
\]
\emph{the grade shift functor}, which is defined as follows. For a $G$-graded $A$-module $X$, 
\begin{list}{$\bullet$}{}
\item $X(i):=X$ as an $A$-module,
\item $G$-grading on $X(i)$ is defined by $X(i)_j:=X_{j+i}$ for any $j \in G$.
\end{list}
This is an autofunctor of $\mod^{G}A$ whose inverse is $(-i)$.

The second one is the $K$-dual.
The standard duality
\[
D := \Hom_{K}(-,K): \mathrm{mod} A \rightarrow \mathrm{mod} A^{\op}
\]
induces the following duality.
For  a $G$-graded $A$-module $X$, we regard $DX$ as a $G$-graded $A^{\op}$-module by defining $(DX)_i:=D(X_{-i})$ for any $i \in G$.
Then we have the duality
\[
D:  \mathrm{mod}^{G}A \rightarrow \mathrm{mod}^{G}A^{\op}.
\]

The following result gives basic properties of $\mod^GA$.

\begin{prop}\label{graded_KS}
$\mod^{G}A$ is a Hom-finite Krull-Schmidt category
\end{prop}

We have the following description of morphism spaces.

\begin{lem}\label{mor_decomp}
Let $X$ and $Y$ be $G$-graded $A$-modules. Then  
\[
\Hom_A(X,Y)= \bigoplus_{i \in G} \Hom_A(X,Y(i))_0.
\]
\end{lem}

We describe projective objects and injective objects in $\mathrm{mod}^{G}A$ under a suitable assumption for $A$.

\begin{prop}\label{cs}
Assume that $J_A= J_{A_0} \oplus \left( \bigoplus_{i \in G \backslash \{0\}}A_i \right)$.
We take a set $\overline{\mathrm{PI}}$ of idempotents of $A_0$ such that $\{ eA_0 \ | \ e \in \overline{\mathrm{PI}} \}$ is a complete list of indecomposable projective $A_0$-modules. 
Then the following assertions hold.
\begin{enumerate}
\def\labelenumi{(\theenumi)} 
\item Any complete set of orthogonal primitive idempotents of $A_0$ is that of $A$.
\item A complete list of simple objects in $\mathrm{mod}^{G}A$ is given by
\[
\{ S(i) \ | \ i \in G, \mbox{$S$ is a simple $A_0$-module} \}.
\]
\item A complete list of indecomposable projective objects in $\mathrm{mod}^{G}A$ is given by
\[
\{ eA(i) \ | \ i \in G,\  e \in \overline{\mathrm{PI}} \}.
\]
\item A complete list of indecomposable injective objects in $\mathrm{mod}^{G}A$ is given by
\[
\{ D(Ae)(i) \ | \ i \in G, \ e \in \overline{\mathrm{PI}} \}.
\]
\end{enumerate}
\end{prop}

Finally we recall that the tensor algebra of two $G$-graded algebras admits a natural $G$-grading.

\begin{df}\label{tesor_grading}
For two $G$-graded algebras $A$ and $B$, we define a $G$-grading on the tensor algebra $A \otimes_KB$ by
\[
(A\otimes_KB)_i := \left\{ \sum a \otimes b \ | \ a \in A_k, \   b \in A_{\ell}, \ k+\ell = i \right\}
\]
for any $i \in G$.

Moreover let $C$ be a $G$-graded algebra, $X$ be a $G$-graded $A^{\op} \otimes_KB$-module and $Y$ be a $G$-graded $B^{\op}\otimes_KC$-module.
We define a $G$-grading on the $A^{\op}\otimes_KC$-module $X \otimes_{B} Y$ by
\[
(X \otimes_{B} Y)_i := \left\{ \sum x \otimes y \ | \ x \in X_k, \   y \in Y_{\ell}, \ k+\ell = i  \right\}
\]
for any $i \in G$.
\end{df}

\subsection{Positively graded algebras and $\mathbb{Z}/a\mathbb{Z}$-graded algebras}
In this subsection, we recall basic facts about positively graded algebras. 
Let $A$ be a positively graded algebra and $a$ a positive integer.

First we regard $A$ as a $\mathbb{Z}/a\mathbb{Z}$-graded algebra by
\[
A_{\overline{i}} := \bigoplus_{j \in i + a\mathbb{Z}} A_j.
\]
for any $i \in \mathbb{Z}$. Throughout this paper,  we often regard a $\mathbb{Z}$-graded algebra as a $\mathbb{Z}/a\mathbb{Z}$-graded algebra in this way.

Also we regard $X \in \mod^{\mathbb{Z}}A$ as a $\mathbb{Z}/a\mathbb{Z}$-graded $A$-module by
\[
X_{\overline{i}} := \bigoplus_{j \in i+a\mathbb{Z}} X_j.
\]
In this way we have an additive functor 
\[
F_a: \mathrm{mod}^{\mathbb{Z}}A \longrightarrow \mathrm{mod}^{\mathbb{Z}/a\mathbb{Z}}A.
\]
In the case that $a=1$, then $\mod^{\mathbb{Z}/a\mathbb{Z}}A$ coincides with $\mod A$, and $F_a$ is just the forgetful functor from $\mod^{\mathbb{Z}}A$ to $\mod A$.

Since $A$ is positively graded and by the above definition, $A$ with both gradings satisfy the assumption of Proposition \ref{cs}.

\begin{prop}\label{pos_rad}
For a positively graded algebra $A$, the equation $J_A= J_{A_0} \oplus \left( \bigoplus_{i \neq 0}A_i \right)= J_{A_{\overline{0}}} \oplus \left( \bigoplus_{\overline{i} \neq \overline{0}}A_{\overline{i}} \right)$ holds. 
In particular the assertions in Proposition \ref{cs} hold.
\end{prop}

We have the following description of morphism spaces in $\mod^{\mathbb{Z}/a\mathbb{Z}}A$ (cf. Lemma \ref{mor_decomp}).

\begin{prop}\label{a}
Let $X$ and $Y$ be $\mathbb{Z}$-graded $A$-modules. Then
\[
\Hom_A(F_a(X),F_a(Y))_{\overline{0}} = \bigoplus_{i \in \mathbb{Z}} \Hom_A(X,Y(ia))_0.
\]
\end{prop}

\subsection{Stable categories of self-injective algebras}
In this subsection, we give basic properties of stable categories of self-injective algebras. 

Let $G$ be an abelian group and $A$ a $G$-graded algebra.
We denote by $\underline{\mod}^GA$ the stable category of $\mod^GA$.
The following observation is clear.

\begin{lem}\label{stable_shift}
For any $i \in G$, the autofunctor $(i): \mod^{G}A \rightarrow \mod^{G}A$ induces an autofunctor $(i) : \underline{\mod}^{G}A \longrightarrow \underline{\mod}^{G}A$.
\end{lem}

For the stable category, the following analogue of Proposition \ref{mor_decomp} holds.

\begin{prop}\label{stablemor_decomp}
Let $X$ and $Y$ be $G$-graded $A$-modules. Then
\[
\underline{\Hom}_A(X,Y)= \bigoplus_{i \in G} \underline{\Hom}_A(X,Y(i))_0.
\]
\end{prop}

In the following, we consider the stable categories of self-injective algebras.
First we recall the definition of self-injective algebras.

\begin{df}\label{df_sym}
An algebra $A$ is called \emph{self-injective} if the following equivalent conditions are satisfied. 
\begin{enumerate}
\def\labelenumi{(\theenumi)} 
\item $A$ is an injective $A$-module.
\item $A$ is an injective $A^{\op}$-module.
\end{enumerate}
Furthermore an algebra $A$ is called \emph{symmetric} if $A$ is isomorphic to $DA$ as a $A^{\op} \otimes_K A$-module.
\end{df}

For a $G$-graded algebra $A$, the stable category $\underline{\mathrm{mod}}^{G}A$ has the \emph{graded syzygy functor} 
\[
\Omega: \underline{\mathrm{mod}}^{G}A \longrightarrow \underline{\mathrm{mod}}^{G}A.
\] 
For $X \in \underline{\mathrm{mod}}^{G}A$, the object $\Omega(X) \in \underline{\mathrm{mod}}^{G}A$ is defined as the kernel of a fixed projective cover of $X$ in $\mathrm{mod}^{G}A$. 
If $A$ is self-injective, this functor is an equivalence, and a quasi-inverse of it is given by the \emph{cosyzygy functor} $\Omega^{-1}: \underline{\mathrm{mod}}^{G}A \rightarrow \underline{\mathrm{mod}}^{G}A$ which is defined as dual of the syzygy functor (cf. \cite[IV 3]{ARS}).

We have the following result by Proposition \ref{pos_rad}, and Happel's general result \cite{Ha1}.

\begin{lem}\cite[Theorem 2.6]{Ha1} \label{Gra-Fro} 
If $A$ is a $G$-graded self-injective algebra, the following assertions hold.
\begin{enumerate}
\def\labelenumi{(\theenumi)} 
\item $\mathrm{mod}^{G}A$ is a Frobenius category.
\item $\underline{\mathrm{mod}}^{G}A$ has a structure of a triangulated category whose shift functor $[1]$ is given by the graded cosyzygy functor $\Omega^{-1}$.
\end{enumerate}
\end{lem}

We fix a positively graded self-injective algebra $A$ and a positive integer $a$. 
Then we have triangulated categories $\underline{\mod}^{\mathbb{Z}}A$ and $\underline{\mod}^{\mathbb{Z}/a\mathbb{Z}}A$. 
The additive functor $F_a:\mathrm{mod}^{\mathbb{Z}}A \rightarrow \mathrm{mod}^{\mathbb{Z}/a\mathbb{Z}}A$ induces an additive functor
\[
\underline{F_a}:\underline{\mathrm{mod}}^{\mathbb{Z}}A \longrightarrow \underline{\mathrm{mod}}^{\mathbb{Z}/a\mathbb{Z}}A.
\]
Since $F_a$ is an exact functor and the diagram
\[
\xymatrix{
\underline{\mathrm{mod}}^{\mathbb{Z}}A \ar[d]_{\underline{F_a}} \ar[r]^{\Omega^{-1}} & \underline{\mathrm{mod}}^{\mathbb{Z}}A \ar[d]^{\underline{F_a}} \\
\underline{\mathrm{mod}}^{\mathbb{Z}/a\mathbb{Z}}A \ar[r]_{\Omega^{-1}} & \underline{\mathrm{mod}}^{\mathbb{Z}/a\mathbb{Z}}A
}
\]
commutes, we have the following immediately.

\begin{lem}\label{lem_tri_fun}
 $\underline{F_a}$ is a triangle-functor.
\end{lem}

We have the following analogue of Proposition \ref{a}.

\begin{prop}\label{Z_to_Z/aZ}
Let $X$ and $Y$ be $\mathbb{Z}$-graded $A$-modules. Then
\[
\underline{\Hom}_A(\underline{F_a}(X),\underline{F_a}(Y))_{\overline{0}} = \bigoplus_{i \in \mathbb{Z}} \underline{\Hom}_A(X,Y(ia))_0.
\]
\end{prop}

We have the following observation.

\begin{prop}\label{sma_tri}
We have  $\thick(\IM \underline{F_a})=\underline{\mathrm{mod}}^{\mathbb{Z}/a\mathbb{Z}}A$.
\end{prop}
\begin{proof}
Any object in $\mathrm{mod}^{\mathbb{Z}/a\mathbb{Z}}A$ has a finite filtration by simple objects in $\mathrm{mod}^{\mathbb{Z}/a\mathbb{Z}}A$.
By Proposition \ref{pos_rad}, $\IM F_a$ contains all simple objects in $\mathrm{mod}^{\mathbb{Z}/a\mathbb{Z}}A$. 
Since short exact sequences in $\mathrm{mod}^{\mathbb{Z}/a\mathbb{Z}}A$ give rise to triangles in $\underline{\mathrm{mod}}^{\mathbb{Z}/a\mathbb{Z}}A$, $\IM \underline{F_a}$ generates $\underline{\mathrm{mod}}^{\mathbb{Z}/a\mathbb{Z}}A$ as a triangulated category.
\end{proof}

We end this section by the following result which follows from the Auslander-Reiten duality.

\begin{prop}\cite[Propositon 1.2]{AR} \label{gr_Serre}
The stable category $\underline{\mod}^{\mathbb{Z}}A$ has a Serre functor.
\end{prop}

\section{Realizing sable categories as derived categories}
\label{cal_section}
The aim of this section is to show Theorem \ref{main_thm}. 
First we give a proof of Theorem \ref{main_thm} (1) $\Leftrightarrow$ (2) $\Leftrightarrow$ (3). 
In the proof of (1) $\Rightarrow$ (2), we construct a tilting object $T$.
Next we study properties of the endomorphism algebra of the tilting object $T$. 
Finally we prove Theorem \ref{main_thm} (1) $\Leftrightarrow$ (4) as an application.
In the proof of (1) $\Rightarrow$ (4), we construct a triangle-equivalence in (4) in two different ways.

\subsection{Existence of tilting objects in stable categories}
\label{existence_tilt_section}

In this  subsection, we give a proof of Theorem \ref{main_thm} (1) $\Leftrightarrow$ (2) $\Leftrightarrow$ (3).
Throughout this section, let $A$ be a positively graded self-injective algebra.

\begin{thm}\label{main_thm2}
Let $A$ be a positively graded self-injective algebra. The following conditions are equivalent.
\begin{enumerate}
\def\labelenumi{(\theenumi)} 
\item $A_0$ has finite global dimension.
\item $\underline{\mathrm{mod}}^{\mathbb{Z}}A$ has tilting objects.
\item There exists a triangle-equivalence $\underline{\mathrm{mod}}^{\mathbb{Z}}A \simeq \mathsf{K}^{\mathrm{b}}(\mathrm{proj}\Lambda)$ for some algebra $\Lambda$.
\end{enumerate}
\end{thm}

In the following, we give a proof of this result.
We have Theorem \ref {main_thm2} (2) $\Leftrightarrow$ (3) by Theorem \ref{Keller}.
We prove Theorem \ref{main_thm2} (2) $\Rightarrow$ (1). 
Our strategy of proof is similar to \cite[Example 2.5 (2)]{AI}.
We need the following easy lemma which gives a property of triangulated categories having tilting objects.

\begin{lem} \cite[Proposition 2.4]{AI} \label{hom_lem}
Let $\mathcal{T}$ be a triangulated category. 
If $\mathcal{T}$ has a tilting object, then $\Hom_{\mathcal{T}}(X,Y[i])=0$ holds for any $X,Y \in \mathcal{T}$ and $|i| >> 0$.
\end{lem}

\begin{proof}[Proof of Theorem \ref{main_thm2} (2) $\Rightarrow$ (1)]
We assume that $\gldim A_0 = \infty$ and $\underline{\mathrm{mod}}^{\mathbb{Z}}A$ has a tilting object $T$. 
First we note that for any $A_0$-module $X$, the degree $0$ part of a projective resolution of $X$ in $\mod^{\mathbb{Z}}A$ gives a projective resolution of $X$ in $\mod A_0$ by Proposition \ref{pos_rad}. 
So we have $\Ext^i_A(A_0/J_{A_0},A_0/J_{A_0})_0=\Ext^i_{A_0}(A_0/J_{A_0},A_0/J_{A_0})$ for any $i \geq 0$.

Next since $\gldim A_0 = \infty$, the $A_0$-module $A_0/J_0$ has infinite projective dimension.
Thus we have $\Ext^i_{A_0}(A_0/J_{A_0},A_0/J_{A_0}) \neq 0$ for any $i \geq 0$.

Consequently we have
\[
\underline{\Hom}_A(A_0/J_{A_0},A_0/J_{A_0}[i])_0=\Ext^i_A(A_0/J_{A_0},A_0/J_{A_0})_0=\Ext^i_{A_0}(A_0/J_{A_0},A_0/J_{A_0}) \neq 0
\]
for any $i \geq 0$. This contradicts to Lemma \ref{hom_lem}.
\end{proof}

In the rest of this subsection, we prove Theorem \ref{main_thm2} (1) $\Rightarrow$ (2) by constructing a tilting object in $\underline{\mathrm{mod}}^{\mathbb{Z}}A$.
We need truncation functors 
\[
(-)_{\geq i}: \mathrm{mod}^{\mathbb{Z}} A \rightarrow \mathrm{mod}^{\mathbb{Z}} A, \ \ \ 
(-)_{\leq i}: \mathrm{mod}^{\mathbb{Z}} A \rightarrow \mathrm{mod}^{\mathbb{Z}} A
\]
which are defined as follows. 
For a $\mathbb{Z}$-graded $A$-module $X$, $X_{\geq i}$ is a $\mathbb{Z}$-graded sub $A$-module of $X$ defined by 
\[
(X_{\geq i})_j := \begin{cases}
0 & (j < i) \\
X_j & (j \geq i),
\end{cases}
\]
and $X_{\leq i}$ is a $\mathbb{Z}$-graded factor $A$-module $X/X_{\geq i+1}$ of $X$.

Now we define a $\mathbb{Z}$-graded $A$-module by
\begin{eqnarray}
T:= \bigoplus_{i \geq 0} A(i)_{\leq 0}. \label{df_T}
\end{eqnarray}
Since $A(i)_{\leq 0}=A(i)$ is zero in $\underline{\mod}^{\mathbb{Z}}A$ for sufficiently large $i$, we can regard $T$ as an object in $\underline{\mathrm{mod}}^{\mathbb{Z}}A$.
Let 
\begin{eqnarray}
\Gamma:=\underline{\End}_A(T)_0. \label{df_Gamma}
\end{eqnarray}
Now we have the following result, where the second assertion implies Theorem \ref{main_thm2} (1) $\Rightarrow$ (2).

\begin{thm}\label{main_thm1}
Under the above setting, the following assertions hold.
\begin{enumerate}
\def\labelenumi{(\theenumi)} 
\item $T$ is a tilting object in the subcategory $\thick T$ of $\underline{\mod}^{\mathbb{Z}}A$. 
\item If $A_0$ has finite global dimension, then $T$ is a tilting object in $\underline{\mathrm{mod}}^{\mathbb{Z}}A$. 
\end{enumerate}
\end{thm}

We prove the above result by checking two conditions in Definition \ref{df_tilting}. 
First we show that the self-extensions of $T$ vanish without assuming that $A_0$ has finite global dimension. 

\begin{lem}\label{van_self-ext}
We have $\underline{\Hom}_A(T,T[i])_0=0$ for any $i \neq 0$.
\end{lem}
\begin{proof}
We take a projective resolution
\[
\cdots \rightarrow P^2 \rightarrow P^1 \rightarrow A(i) \rightarrow A(i)_{\leq 0} \rightarrow 0
\]
of $A(i)_{\leq 0}$ in $\mod^{\mathbb{Z}}A$. Since $A$ is positively graded, we have $(P^j)_{\leq 0}=0$ for $j>0$. 
Thus $(\Omega^{j}T)_{\leq 0}=0$ holds for  any $j>0$. 
Since $T=T_{\leq 0}$, we have $\Hom_{A}(T,\Omega^{j}T)_0=0=\Hom_A(\Omega^{j}T,T)_0$ for any $j>0$.
Consequently we have 
\[
\underline{\Hom}_A(T,T[i])_0=\underline{\Hom}_A(\Omega^iT,T)_0=0,
\]
\[
\underline{\Hom}_A(T,T[-i])_0 = \underline{\Hom}_A(T,\Omega^{i}T)_0=0
\]
for $i>0$.
\end{proof}

Theorem \ref{main_thm1} (1) follows from Lemma \ref{van_self-ext}.

Next we  prove the following result. 

\begin{lem}\label{gen_T}
If $A_0$ has finite global dimension, then we have $\underline{\mathrm{mod}}^{\mathbb{Z}}A=\thick T$.
\end{lem}
\begin{proof}
We regard $A_0$ as the natural $\mathbb{Z}$-graded factor $A$-module of $A$, i.e. $A_0(0)=A(0)_{\leq 0}$.
Any object in $\mathrm{mod}^{\mathbb{Z}}A$ has a finite filtration by simple objects in $\mod^{\mathbb{Z}}A$ which are given by simple $A_0$-modules concentrated in some degree by Proposition \ref{pos_rad}.
Since $A_0$ has finite global dimension, it is enough to show that $A_0(i) \in \thick T$ for any $i \in \mathbb{Z}$. 
We divide the proof into two parts.

(i) We show $A_0(i) \in \thick T$ for any $i \geq 0$ by induction on $i$. 
Obviously we have $A_0(0)=A(0)_{\leq 0} \in \thick T$. 
We assume $A_0(0),\cdots,A_0(i-1) \in \thick T$. 
There exists an exact sequence
\[
0 \rightarrow (A(i)_{\leq 0})_{\geq 1-i} \rightarrow A(i)_{\leq 0} \rightarrow A_0(i) \rightarrow 0,
\]
in $\mathrm{mod}^{\mathbb{Z}}A$. 
By the inductive hypothesis, we have $(A(i)_{\leq 0})_{\geq 1-i} \in \thick T$.
Thus we have $A_0(i) \in \thick T$.

(ii) We show that $A_0(-i) \in \thick T$ for any $i \geq 1$.
We assume $A_0(-j) \in \thick T$ for any $j <i$.
We put $n:=\mathrm{pd}_{A^{\op}_0} D(A_0)+1$. Then there exists an exact sequence
\[
0 \rightarrow X \rightarrow Q^{n-1} \rightarrow \cdots \rightarrow Q^1 \rightarrow Q^0 \rightarrow D(A_0)\rightarrow 0
\]
in $\mathrm{mod}^{\mathbb{Z}}A^{\op}$ such that $Q^j$ is a $\mathbb{Z}$-graded projective $A^{\op}$-module for $0 \leq j \leq n-1$, and $X_{\leq 0} =0$.
Thus we have an exact sequence
\[
0 \rightarrow A_0 \rightarrow D(Q^0) \rightarrow D(Q^1) \rightarrow \cdots \rightarrow D(Q^{n-1}) \rightarrow D(X) \rightarrow 0
\]
in $\mathrm{mod}^{\mathbb{Z}}A$ such that $D(Q^j)$ is a $\mathbb{Z}$-graded projective $A$-modules for $0 \leq j \leq n-1$, and $(D(X))_{\geq 0}=0$. 
We have $A_0(-i) = D(X)(-i)[-n]$ in $\underline{\mod}^{\mathbb{Z}}A$.
Since $(D(X)(-i))_{\geq i}=0$, we have $D(X)(-i) \in \thick T$ by the inductive hypothesis. 
Thus we have $A_0(-i) \in \thick T$.
\end{proof}

Now we finished the proof of Theorem \ref{main_thm1}.
\hfill{$\square$}

\subsection{Calculation of the endomorphism algebra of the tilting object}
\label{calculation_section}

We keep the notations in the previous subsection.
The aim of this subsection is to calculate $\Gamma$ given in \eqref{df_Gamma}.
Although $\Gamma$ is defined as a Hom-space in $\underline{\mod}^{\mathbb{Z}}A$, we can calculate $\Gamma$ as a Hom-space in $\mod^{\mathbb{Z}}A$ by removing projective direct summands from $T$.

\begin{prop}\label{cal_end}
Take a decomposition $T=\underline{T} \oplus P$ in $\Mod^{\mathbb{Z}}A$ where $\underline{T}$ is a direct sum of all indecomposable non-projective direct summands of $T$.
Then the following assertions hold.
\begin{enumerate}
\def\labelenumi{(\theenumi)} 
\item $\underline{T}$ is finitely generated.
\item $\underline{T}$ is isomorphic to $T$ in $\underline{\mathrm{mod}}^{\mathbb{Z}}A$.
\item $\Hom_A(Q,\underline{T})_0=0$ holds for any projective object $Q$ in $\mod^{\mathbb{Z}}A$ satisfying $\Soc Q \subset Q_{\ell}$ for some $\ell \leq 0$.
\item There exists an algebra isomorphism $\Gamma \simeq \End_A(\underline{T})_0$.
\end{enumerate}
\end{prop}
\begin{proof}
(1), (2) The assertions are obvious.

(3) Let $Q^1$ and $Q^2$ be indecomposable projective objects in $\mod^{\mathbb{Z}}A$. 
Since $Q^1$ and $Q^2$ have simple socle as $A$-modules, there exist integers $\ell_1$ and $\ell_2$ such that $\Soc Q^1 \subset Q^1_{\ell_1}$ and $\Soc Q^2 \subset Q^2_{\ell_2}$.
Then if $\ell_1<\ell_2$, we have $\Hom_A(Q^1,Q^2)_0=0$ since $Q^1_{\ell_2}=0$ and $\Soc Q^2 \subset Q^2_{\ell_2}$.

Now we show the assertion. 
Let $Q$ be a projective object in $\mod^{\mathbb{Z}}A$ satisfying $\Soc Q \subset Q_{\ell}$ for some $\ell \leq 0$.
We take a projective cover $Q' \rightarrow \underline{T}$ of $\underline{T}$ in $\mod^{\mathbb{Z}}A$. 
From our construction of $T$ and $\underline{T}$, we have $\Soc Q' \subset Q'_{\geq 1}$. 
Thus by the above observation, we have $\Hom_A(Q,Q')_0=0$. 
Hence $\Hom_A(Q,\underline{T})_0=0$ holds.

(4) Since there exist algebra isomorphisms
\[
\underline{\End}_A(T)_0 \simeq \left( \begin{array}{cc}
\underline{\End}_A(\underline{T})_0 & \underline{\Hom}_A(P,\underline{T})_0 \\
\underline{\Hom}_A(\underline{T},P)_0 & \underline{\End}_A(P)_0 
\end{array} \right) = \left( \begin{array}{cc}
\underline{\End}_A(\underline{T})_0 & 0 \\
0 & 0
\end{array} \right),
\]
it is enough to show that $\underline{\End}_A(\underline{T})_0=\End_A(\underline{T})_0$.

Let $Q$ be an indecomposable projective object in $\mod^{\mathbb{Z}}A$. 
Then there exists $\ell \in \mathbb{Z}$ such that $\Soc Q \subset Q_{\ell}$.
First if $\ell \leq 0$, then by (3), we have $\Hom_{A}(Q,\underline{T})_0=0$.
Next if $\ell > 0$, then we have $\Hom_{A}(\underline{T},Q)_0=0$ since $\underline{T}_{\ell}=0$.
Consequently there exists no non-zero morphisms from $\underline{T}$ to $\underline{T}$ factoring through projective objects in $\mathrm{mod}^{\mathbb{Z}}A$. 
Thus we have $\underline{\End}_A(\underline{T})_0=\End_A(\underline{T})_0$.
\end{proof}

\subsection{Global dimension of the endomorphism algebra of the tilting object}
\label{global_section}

We keep the notations in the previous subsections.
The aim of this subsection is to show the following result.

\begin{thm}\label{gldim_fin}
If $A_0$ has finite global dimension, then so does $\Gamma$.
\end{thm}

In the following, we prove Theorem \ref{gldim_fin}.
We need the following observations.

\begin{lem}\cite[III Proposition 2.7]{ARS}\label{lem_ARS}
Let $\Lambda$ and $\Gamma$ be algebras and $M$ a $\Lambda^{\op}\otimes_K\Gamma$-module. 
Then the algebra
\[
\left(\begin{array}{cc}
\Lambda & M \\
0 & \Gamma
\end{array}\right)
\]
has finite global dimension if and only if so do $\Lambda$ and $\Gamma$.
\end{lem}

\begin{lem}\label{nonstab_end}
For a positive integer $\ell$, let $U:= \bigoplus^{\ell-1}_{i=0}A(i)_{\leq 0}$. 
Then the following assertions hold. 
\begin{enumerate}
\def\labelenumi{(\theenumi)} 
\item There exists an algebra isomorphism 
\[
\End_A(U)_0 \simeq \left( \begin{array}{ccccc} 
A_0 & A_1 & \cdots & A_{\ell-2} &  A_{\ell-1}  \\
 & A_0 & \cdots & A_{\ell-3} & A_{\ell-2}  \\
 &  & \ddots & \vdots & \vdots \\
 &  &  & A_0 & A_1   \\
0 &  & & &  A_0 
\end{array} \right).
\]
\item The algebra $A_0$ has finite global dimension if and only if so does $\End_A(U)_0$.
\end{enumerate}
\end{lem}
\begin{proof}
(1) Since $\Hom_A(A(i)_{\leq 0},A(j)_{\leq 0})_0=A_{j-i}$, we have the assertion. 

(2) The assertion follows from (1) and Lemma \ref{lem_ARS}.
\end{proof}

Now we can prove Theorem \ref{gldim_fin}.

\begin{proof}[Proof of Theorem \ref{gldim_fin}]
We take a positive integer $\ell$ such that $A=A_{\leq \ell}$.
Then $A(i)_{\leq 0}$ is a projective object in $\mod^{\mathbb{Z}}A$ for any $i \geq \ell$. 
Let $U:= \bigoplus^{\ell-1}_{i=0}A(i)_{\leq 0}$.
Then $\End_A(U)_0$ has finite global dimension by Lemma \ref{nonstab_end} (2).
We can decompose $U=\underline{T} \oplus P'$ where $\underline{T}$ is the direct summand of $T$ given in Proposition \ref{cal_end} and $P'$ is a projective direct summand of $U$.
By Proposition \ref{cal_end} (3) and (4), we have
\[
\End_A(U)_0 \simeq \left( \begin{array}{cc}
\End_A(P')_0 & \Hom_A(\underline{T},P')_0 \\
\Hom_A(P',\underline{T})_0 & \End_A(\underline{T})_0
\end{array} \right) \simeq \left( \begin{array}{cc}
\End_A(P')_0 & \Hom_A(\underline{T},P')_0 \\
0 & \Gamma
\end{array} \right).
\]
The assertion follows from Lemma \ref{lem_ARS}.
\end{proof}

\subsection{Existence of a triangle-equivalnce}
\label{tri_equ_section}

We keep the notations in the previous subsections.
In this subsection, we prove Theorem \ref{main_thm} (1) $\Leftrightarrow$ (4) by applying results in previous subsections.

\begin{thm}\label{main_thm3}
Let $A$ be a positively graded self-injective algebra.
The following conditions are equivalent. 
\begin{enumerate}
\def\labelenumi{(\theenumi)} 
\item $A_0$ has finite global dimension.
\item There exists a triangle-equivalence $\underline{\mathrm{mod}}^{\mathbb{Z}}A \simeq \mathsf{D}^{\mathrm{b}}(\mathrm{mod}\Lambda)$ for some algebra $\Lambda$.
\end{enumerate}
\end{thm}

The implication Theorem \ref{main_thm3} (1) $\Rightarrow$ (2) follows from the following observation.

\begin{thm}\label{equ1}
The following assertions hold.
\begin{enumerate}
\def\labelenumi{(\theenumi)}
\item There exists a triangle-equivalence 
\[
\mathrm{thick} T \longrightarrow \mathsf{K}^{\mathrm{b}}(\mathrm{proj}\Gamma).
\]
\item If $A_0$ has finite global dimension, then there exists a triangle-equivalence 
\[
\underline{\mathrm{mod}}^{\mathbb{Z}}A \longrightarrow \mathsf{D}^{\mathrm{b}}(\mathrm{mod}\Gamma).
\]
\end{enumerate}
\end{thm}
\begin{proof}
(1) The assertion is immediate from Theorems \ref{Keller} and \ref{main_thm1} (1).

(2) We assume that $A_0$ has finite global dimension.
First by Theorem \ref{Keller} and Theorem \ref{main_thm1} (2), there exists the triangle-equivalence
$\underline{\mathrm{mod}}^{\mathbb{Z}}A \rightarrow \mathsf{K}^{\mathrm{b}}(\mathrm{proj}\Gamma)$.
Moreover $\Gamma$ has finite global dimension by Theorem \ref{gldim_fin}.
By Theorems \ref{Serre=gldim},  the canonical embedding  $\mathsf{K}^{\mathrm{b}}(\mathrm{proj}\Gamma) \rightarrow \mathsf{D}^{\mathrm{b}}(\mathrm{mod}\Gamma)$ is an equivalence.
Composing these equivalences,  we have the desired triangle-equivalence.
\end{proof}

The implication Theorem \ref{main_thm3} (1) $\Rightarrow$ (2) follows from Theorems \ref{equ1} (2).
Now we prove Theorem \ref{main_thm3} (2) $\Rightarrow$ (1).

\begin{proof}[Proof of Theorem \ref{main_thm3} (2) $\Rightarrow$ (1)]
Assume that there exists a triangle-equivalence $\underline{\mathrm{mod}}^{\mathbb{Z}}A \simeq\mathsf{D}^{\mathrm{b}}(\mathrm{mod}\Lambda)$ for some algebra $\Lambda$. 
By Theorem \ref{gr_Serre}, $\mathsf{D}^{\mathrm{b}}(\mathrm{mod}\Lambda)$ has a Serre functor. 
By Theorem \ref{Serre=gldim}, we have a triangle-equivalence $\mathsf{K}^{\mathrm{b}}(\mathrm{proj}\Lambda) \simeq \mathsf{D}^{\mathrm{b}}(\mathrm{mod}\Lambda)$. 
Consequently we have a triangle-equivalence $\underline{\mathrm{mod}}^{\mathbb{Z}}A \simeq \mathsf{K}^{\mathrm{b}}(\mathrm{proj}\Lambda)$.
By Theorem \ref{main_thm2}, $A_0$ has finite global dimension.
\end{proof}

Since finiteness of global dimension is preserved under derived equivalences, we have the following result from Theorem \ref{main_thm3}.

\begin{cor}
If $A_0$ has finite global dimension, then so does the endomorphism algebra of arbitrary tilting object in $\underline{\mod}^{\mathbb{Z}} A$. 
\end{cor}

This result implies the last assertion of Theorem \ref{main_thm}. 
Thus we finish the proof of Theorem \ref{main_thm}.
\hfill{$\square$}

\subsection{Direct construction of a triangle-equivalence}
\label{direct_section}

When $A_0$ has finite global dimension, we gave a triangle-equivalence $\underline{\mod}^{\mathbb{Z}}A \simeq \mathsf{D}^{\mathrm{b}}(\mod\Gamma)$ in Theorem \ref{equ1}. 
However the construction heavily depends on Theorem \ref{Keller} which is shown by using a lot of techniques in differential graded algebras. 
In this subsection, we give a direct construction of the equivalence without using Theorem \ref{Keller}.
We keep the notations in the previous subsections.

We consider the decomposition $T= \underline{T} \oplus P$ given in Proposition \ref{cal_end}. 
Since $\Gamma \simeq \End_A(\underline{T})_0$, $\Gamma$ acts on $\underline{T}$ from the left naturally.
By this action, we can regard $\underline{T}$ as a $\mathbb{Z}$-graded $\Gamma^{\op} \otimes_K A$-module (see Definition \ref{tesor_grading} for grading on $\Gamma^{\op} \otimes_K A$). 
Thus we have the left derived tensor functor 
\[
-\Lten_{\Gamma}\underline{T}:\mathsf{D}^{\mathrm{b}}(\mod \Gamma) \longrightarrow \mathsf{D}^{\mathrm{b}}(\mod^{\mathbb{Z}}A).
\]

Now we consider the quotient category (cf.\cite{Har,Ne}) $\mathsf{D}^{\mathrm{b}}(\mod^{\mathbb{Z}} A)/\mathsf{K}^{\mathrm{b}}(\proj^{\mathbb{Z}} A)$ and the canonical functor
\begin{eqnarray}
\mathsf{D}^{\mathrm{b}}(\mod^{\mathbb{Z}} A) \longrightarrow \mathsf{D}^{\mathrm{b}}(\mod^{\mathbb{Z}} A)/\mathsf{K}^{\mathrm{b}}(\proj^{\mathbb{Z}} A). \label{Quotient}
\end{eqnarray}

The following triangle-equivalence is a realization of $\underline{\mod}^{\mathbb{Z}}A$ as the quotient category $\mathsf{D}^{\mathrm{b}}(\mod^{\mathbb{Z}} A)/\mathsf{K}^{\mathrm{b}}(\proj^{\mathbb{Z}} A)$.

\begin{thm}\label{Rickard} \cite[Theorem 2.1]{Ric2}
Let $A$ be a $\mathbb{Z}$-graded self-injective algebra. 
The canonical embedding $\mod^{\mathbb{Z}}A \rightarrow \mathsf{D}^{\mathrm{b}}(\mod^{\mathbb{Z}}A)$ induces a triangle-equivalence 
\[
\underline{\mod}^{\mathbb{Z}}A \longrightarrow \mathsf{D}^{\mathrm{b}}(\mod^{\mathbb{Z}} A)/\mathsf{K}^{\mathrm{b}}(\proj^{\mathbb{Z}} A).
\]
\end{thm}

Now we consider the composition 
\begin{eqnarray}
H: \mathsf{D}^{\mathrm{b}}(\mod^{\mathbb{Z}}A) \longrightarrow \mathsf{D}^{\mathrm{b}}(\mod^{\mathbb{Z}} A)/\mathsf{K}^{\mathrm{b}}(\proj^{\mathbb{Z}} A) \longrightarrow \underline{\mod}^{\mathbb{Z}}A. \label{df_H}
\end{eqnarray}
of the canonical functor and a quasi-inverse of the triangle-equivalence given in Theorem \ref{Rickard}.
Moreover we consider the composition
\begin{eqnarray*}
G:\mathsf{D}^{\mathrm{b}}(\mod \Gamma) \xrightarrow{\ -\Lten_{\Gamma}\underline{T}\ } \mathsf{D}^{\mathrm{b}}(\mod^{\mathbb{Z}}A) 
\xrightarrow{\ H\ } \underline{\mod}^{\mathbb{Z}}A. 
\end{eqnarray*}

\begin{thm}\label{equ2}
Under the above notations, the following assertions hold.
\begin{enumerate}
\def\labelenumi{(\theenumi)} 
\item $G$ induces a triangle-equivalence $\mathsf{K}^{\mathrm{b}}(\proj\Gamma) \rightarrow \thick T$.
\item If $A_0$ has finite global dimension, then $G$ is a triangle-equivalence.
\end{enumerate}
\end{thm}
\begin{proof}
(1) Clearly $G(\Gamma)$ is isomorphic to $\underline{T}$, which is isomorphic to $T$. Moreover $G$ induces an isomorphism 
\[
\mathrm{Hom}_{\mathsf{D}^{\mathrm{b}}(\mathrm{mod} \Gamma)}(\Gamma,\Gamma[i]) \simeq \underline{\mathrm{Hom}}_A(G(\Gamma),G(\Gamma)[i])_0
\]
for any $i \in \mathbb{Z}$. Indeed we have $\mathrm{Hom}_{\mathsf{D}^{\mathrm{b}}(\mathrm{mod} \Gamma)}(\Gamma,\Gamma[i]) =0= \underline{\mathrm{Hom}}_A(G(\Gamma),G(\Gamma)[i])_0$ for any $i \neq 0$, and $\mathrm{End}_{\mathsf{D}^{\mathrm{b}}(\mathrm{mod} \Gamma)}(\Gamma)\simeq \Gamma \simeq \End_A(\underline{T})_0 \simeq \underline{\mathrm{End}}_A(G(\Gamma))_0$ by Proposition \ref{cal_end} (4).

By this and Lemma \ref{uses_five} (1), $G$ is fully faithful on $\mathsf{K}^{\mathrm{b}}(\mathrm{proj} \Gamma)$.
Thus $G$ induces a triangle-equivalence $\mathsf{K}^{\mathrm{b}}(\mathrm{proj} \Gamma) \rightarrow \mathrm{thick} T$.

(2) If $A_0$ has finite global dimension, then $\mathrm{thick}T=\underline{\mathrm{mod}}^{\mathbb{Z}}A$ by Theorem \ref{main_thm1} (2).
Also the canonical embedding $\mathsf{K}^{\mathrm{b}}(\mathrm{proj} \Gamma) \rightarrow \mathsf{D}^{\mathrm{b}}(\mod \Gamma)$ is an equivalence by Lemma \ref{Serre=gldim} and Theorem  \ref{gldim_fin}.
Thus by (1), $G$ is an equivalence. 
\end{proof}

\subsection{Examples}
In this subsection, we give some examples and applications of results in previous subsections. 

First application is famous Happel's equivalence \eqref{intro_equation} \cite{Ha2}, which  gives a relationship between representation theory of algebras and that of the trivial extensions. 
We show it as a consequence of our results.

\begin{ex} \cite[Theorem 2.3]{Ha2}\label{equ_Happel}
Let $\Lambda$ be an algebra. 
The \emph{trivial extension} $A$ of $\Lambda$ is defined as follows.
\begin{list}{$\bullet$}{}
\item $A=\Lambda \oplus D\Lambda$ as an abelian group.
\item The multiplication in $A$ is given by
\[
(x,f) \cdot (y,g):= (xy,xg+fy).
\]
for any $x, y \in \Lambda$ and $f,g \in D\Lambda$.
Here $xg$ and $fy$ are defined by the natural $\Lambda^{\op} \otimes_K \Lambda$-module structure on $D\Lambda$.
\end{list}
Then $A$ is an algebra, and moreover it is easy to check that $A$ is symmetric (see Definition \ref{df_sym}). 
Furthermore $A$ has a structure of a positively graded algebra defined by
\[
A_i := \begin{cases}
\Lambda & (i=0), \\
D\Lambda & (i=1), \\
0 & (i \geq 2).
\end{cases}
\]

Under this setting, $\Lambda$ has finite global dimension if and only if there exists a triangle-equivalence
\[
\underline{\mathrm{mod}}^{\mathbb{Z}}A \simeq \mathsf{D}^{\mathrm{b}}(\mod \Lambda).
\]

\begin{proof}
First we assume that $\Lambda$ has finite global dimension.
Let $T$ be the tilting object constructed in \eqref{df_T}, and $\underline{T}$ the direct summand of $T$ given in Proposition \ref{cal_end}.
Then we have $\underline{T}=\Lambda$, and so $\underline{\End}_A(T)_0 = \End_A(\Lambda)_0 \simeq \Lambda$ by Proposition \ref{cal_end}.
Thus by Theorem \ref{equ1}, we have the desired triangle-equivalence. 

The converse follows from Theorem \ref{main_thm3}.
\end{proof}

\end{ex}

In the second application, we study positively graded self-injective algebras which have Gorenstein parameters. 
As a consequence of Theorem \ref{equ1}, we have Chen's \cite{Ch} result.

\begin{ex}\cite[Corollary 1.2]{Ch} \label{equ_Chen}
Let $A$ be a positively graded self-injective algebra of Gorenstein parameter $\ell$ (see Definition \ref{df_Gor_par}).
Let
\[
\Gamma:=\left( \begin{array}{ccccc} 
A_0 & A_1 & \cdots & A_{\ell-2} &  A_{\ell-1}  \\
 & A_0 & \cdots & A_{\ell-3} & A_{\ell-2}  \\
 &  & \ddots & \vdots & \vdots \\
 &  &  & A_0 & A_1   \\
0 &  & & &  A_0 
\end{array} \right).
\]
Then $A_0$ has finite global dimension if and only if there exists a triangle-equivalence 
\[
\underline{\mod}^{\mathbb{Z}}A \simeq \mathsf{D}^{\mathrm{b}}(\mod \Gamma).
\]

\begin{proof}
Let $T$ be the object defined in \eqref{df_T}, and $\underline{T}$ the direct summand of $T$ defined in Proposition \ref{cal_end}.
Clearly we have $\underline{T}=\bigoplus_{i=0}^{\ell-1}A(i)_{\leq 0}$ and $\underline{\End}_A(T)_0 \simeq \End_A(\underline{T})_0 \simeq \Gamma$ by Propositions \ref{cal_end} and \ref{nonstab_end}.
Then the assertion follows immediately from Theorem \ref{equ1}.
\end{proof}

\begin{rmk}\label{Happel<Chen}
The trivial extensions are positively graded symmetric algebras of Gorenstein parameter $1$. 
Thus Example \ref{equ_Happel} is a special case of Example \ref{equ_Chen}.
\end{rmk}
\end{ex}

Next we give two more concrete examples.

\begin{ex}\label{ex_concrete1}
Let $n$ be a positive integer. 
Let $A=K[x]/(x^{n+1})$ be a positively graded algebra with $\deg x := 1$.
Then $A$ is symmetric and has Gorenstein parameter $n$. 

Since global dimension of $A_0=K$ is zero, the category $\underline{\mod}^{\mathbb{Z}}A$ has a tilting object $T$ given in \eqref{df_T} by Theorem \ref{main_thm1} (2).
By Theorem \ref{equ1}, there exists a triangle-equivalence
\begin{eqnarray}
\underline{\mod}^{\mathbb{Z}}A \simeq \mathsf{D}^{\mathrm{b}}(\mod \Gamma) \label{ex_equ}
\end{eqnarray}
where $\Gamma=\underline{\End}_A(T)_0$.
 By Example \ref{equ_Chen}, the algebra $\Gamma$ is isomorphic to the $n \times n$ upper triangular matrix algebra over $K$.

Let us consider the special case $n=2$ and $A=K[x]/(x^3)$.
For $i=1,2$, let $X^i$ be the $\mathbb{Z}$-graded ideal $(x^i)/(x^{3})$ of $A$.
Then we have a chain $A \supset X^1 \supset X^2$ of $\mathbb{Z}$-graded ideals of $A$.
It is known that $\{ X^i(j) \ | \ i=1,2, \ j \in \mathbb{Z} \}$ is a complete set of indecomposable non-projective $\mathbb{Z}$-graded $A$-modules, and the Auslander-Reiten quiver of $\underline{\mod}^{\mathbb{Z}}A$ is given as follows (cf. \cite[Chapter V]{ASS}).
\begin{eqnarray}
\begin{xy}
(-36,12) *{X^1(-1)}="B" ,
(-48,0) *{X^2(-1)}="A" ,
(-12,12) *{X^1}="D" ,
(-24,0) *{X^2}="C" ,
(12,12) *{X^1(1)}="F" ,
(0,0) *{X^2(1)}="E" ,
(36,12) *{X^1(2)}="H" ,
(24,0) *{X^2(2)}="G" ,
(60,12) *{X^1(3)}="J" ,
(48,0) *{X^2(3)}="I" ,
(-60,6) *{\cdots\cdots},
(72,6) *{\cdots\cdots},

\ar "A" ; "B"
\ar "B" ; "C"
\ar "C" ; "D"
\ar "D" ; "E"
\ar "E" ; "F"
\ar "F" ; "G"
\ar "G" ; "H"
\ar "H" ; "I"
\ar "I" ; "J"
\ar@{.>} "C" ; "A"
\ar@{.>} "E" ; "C"
\ar@{.>} "G" ; "E"
\ar@{.>} "I" ; "G"
\ar@{.>} "D" ; "B"
\ar@{.>} "F" ; "D"
\ar@{.>} "H" ; "F"
\ar@{.>} "J" ; "H"
\end{xy} \label{AR1}
\end{eqnarray}
Here dotted arrows show the Auslander-Reiten translation in $\underline{\mod}^{\mathbb{Z}}A$.

Next we write the Auslander-Reiten quiver of $\mathsf{D}^{\mathrm{b}}(\mod \Gamma)$.
In this case, $\Gamma=\underline{\End}_A(T)_0$ is isomorphic to the $2 \times 2$ upper triangular matrix algebra over $K$.
Define $\Gamma$-modules by $P^1:= (K \ K)$, $P^2:=(0 \ K)$ and $I^1:=(K \ 0)$.
The set $\{P^1,P^2,I^1\}$ is a complete set of indecomposable $\Gamma$-modules, 
and the Auslander-Reiten quiver of $\mathsf{D}^{\mathrm{b}}(\mod \Gamma)$ is given as follows (cf. \cite[Chapter I.5]{Ha1}).
\begin{eqnarray}
\begin{xy}
(-36,12) *{P^2[-1]}="B" ,
(-48,0) *{I^1[-2]}="A" ,
(-12,12) *{I^1[-1]}="D" ,
(-24,0) *{P^1[-1]}="C" ,
(12,12) *{P^1}="F" ,
(0,0) *{P^2}="E" ,
(36,12) *{P^2[1]}="H" ,
(24,0) *{I^1}="G" ,
(60,12) *{I^1[1]}="J" ,
(48,0) *{P^1[1]}="I" ,
(-60,6) *{\cdots\cdots},
(72,6) *{\cdots\cdots},

\ar "A" ; "B"
\ar "B" ; "C"
\ar "C" ; "D"
\ar "D" ; "E"
\ar "E" ; "F"
\ar "F" ; "G"
\ar "G" ; "H"
\ar "H" ; "I"
\ar "I" ; "J"
\ar@{.>} "C" ; "A"
\ar@{.>} "E" ; "C"
\ar@{.>} "G" ; "E"
\ar@{.>} "I" ; "G"
\ar@{.>} "D" ; "B"
\ar@{.>} "F" ; "D"
\ar@{.>} "H" ; "F"
\ar@{.>} "J" ; "H"
\end{xy} \label{AR2}
\end{eqnarray}
Here dotted arrows show the Auslander-Reiten translation in $\mathsf{D}^{\mathrm{b}}(\mod\Gamma)$.

The above Auslander-Reiten quivers of $\underline{\mod}^{\mathbb{Z}}A$ and $\mathsf{D}^{\mathrm{b}}(\mod \Gamma)$ have the same shape. 
This is a consequence of our triangle-equivalence Theorem \ref{equ1}.

\end{ex}

\begin{ex}\label{ex_concrete2}
Let $n$ be a positive integer and $A=K[x_1,\cdots,x_n]/(x^2_i \ | \ 1 \leq i \leq n)$.
We define a grading on $A$ by $\deg x_i := 1$ for any $1 \leq i \leq n$.
Then $A$ is a positively graded symmetric algebra of Gorenstein parameter $n$. 

Since global dimension of $A_0=K$ is zero, Theorem \ref{main_thm2} shows that $\underline{\mod}^{\mathbb{Z}}A$ has a tilting object $T$ given in \eqref{df_T}. 
By Theorem \ref{equ1},  we have a triangle-equivalence
\[
\underline{\mod}^{\mathbb{Z}}A \simeq  \mathsf{D}^{\mathrm{b}}(\mod \Gamma).
\]
where $\Gamma=\underline{\End}_A(T)_0$.
By Example \ref{equ_Chen}, the algebra $\Gamma$ is given by the following quiver with relations.
\[
\xymatrix{
1\ar@<1ex>[r]^{\alpha^1_1}_{\cdot} \ar@<-1ex>[r]^{\cdot}_{\alpha^1_{n}} & 2 \ar@<1ex>[r]^{\alpha^2_{1} \ \ \ }_{\cdot \ \ \ } \ar@<-1ex>[r]^{\cdot \ \ \ }_{\alpha^2_{n} \ \ \ } &  \cdots \cdots \ar@<1ex>[r]^{\alpha^{n-2}_{1}}_{\cdot} \ar@<-1ex>[r]^{\cdot}_{\alpha^{n-2}_{n}} & n-1 \ar@<1ex>[r]^{\ \ \ \alpha^{n-1}_{1}}_{\ \ \ \cdot} \ar@<-1ex>[r]^{\ \ \ \cdot}_{\ \ \ \alpha^{n-1}_{n}} & n
}
\]
\[
\begin{cases}
\alpha^i_{j}\alpha^{i+1}_j=0 & (1 \leq i \leq n-2,\ 1 \leq j \leq n) \\
\alpha^i_j \alpha^{i+1}_k = \alpha^i_k \alpha^{i+1}_{j} & (1 \leq i \leq n-2,\ 1 \leq j,k \leq n)
\end{cases}
\]

\end{ex}

\section{Application to higher preprojective algebras}
\label{app_pp_section}

In this section, we apply our results to a certain class of graded self-injective algebras called higher preprojective algebras. 
In particular we give another proof of \cite[Theorem 4.7]{IO2} in our context.

Throughout this section, we fix a positive integer $n$.

\subsection{Main result}

First we give basic definitions of higher preprojective algebras and higher represetation-finite algebras. 
After that we give our main result in this section.

\begin{df}\label{df_pp}
Let $\Lambda$ be an algebra with global dimension at most $n$. 
Then the tensor algebra
\[
\Pi=\Pi_{n+1}(\Lambda) := \mathrm{T}_{\Lambda}(\Ext^n_{\Lambda}(D\Lambda,\Lambda))
\]
of the $\Lambda^{\op} \otimes_K \Lambda$-module $\Ext^n_{\Lambda}(D\Lambda,\Lambda)$ is called the \emph{$(n+1)$-preprojective algebra} of $\Lambda$.

This is naturally regarded as a graded algebra, i.e. the degree $i$ part $\Pi_i$ is the $i$-th tensor product $\Ext^n_{\Lambda}(D\Lambda,\Lambda) \otimes_{\Lambda} \cdots \otimes_{\Lambda} \Ext^n_{\Lambda}(D\Lambda,\Lambda)$.
\end{df}

\begin{rmk}
The $(n+1)$-preprojective algebras are the $0$-th homology of derived $(n+1)$-preprojective algebras which are introduced in \cite{KVdB}.
\end{rmk}

The homological behavior of the $(n+1)$-preprojective algebra $\Pi$ becomes very nice if we assume that $\Lambda$ is an $n$-representation finite algebra defined as follows.

\begin{df}\cite{Iya2,IO1}\label{df_rep_fin}
An algebra $\Lambda$ is called \emph{$n$-representation finite} if it satisfies the following conditions.
\begin{enumerate}
\def\labelenumi{(\theenumi)} 
\item $\mathrm{gl.dim}\Lambda \leq n$.
\item There exists a $\Lambda$-module $M$ satisfying 
\begin{eqnarray*}
\add M &=& \{ X \in \mathrm{mod}\Lambda \ | \ \mbox{$\mathrm{Ext}^{i}_{\Lambda}(M,X)=0$ for any $1 \leq i \leq n-1$} \} \\
&=&  \{ X \in \mathrm{mod}\Lambda \ | \ \mbox{$\mathrm{Ext}^i_{\Lambda}(X,M)=0$ for any $1 \leq i \leq n-1$} \}.
\end{eqnarray*}
We call such $M$ an \emph{$n$-cluster tilting $\Lambda$-module}.
\end{enumerate}
\end{df}

\begin{ex}\label{ex_pp}
We explain our conditions for the case $n=1$.
The condition (1) above means that $\Lambda$ is hereditary. 
The condition (2) above means that there exists a $\Lambda$-module $M$ such that $\add M=\mod \Lambda$.
Consequently $1$-representation finite algebras are nothing but representation finite hereditary algebras.
By Gabriel's theorem (see \cite[Chapter VII, Theorem 1.7, Theorem 5.10]{ASS}), they are Morita equivalent to path algebras of Dynkin quivers.
\end{ex}

The following gives a basic property of $\Pi$.

\begin{prop} \cite[Corollary 3.4]{IO2} \label{pp_selfinj}
Let $\Lambda$ be an $n$-representation finite algebra. 
Then $\Pi$ is a graded self-injective algebra.
\end{prop}

Now we apply Theorem \ref{equ1} to $\Pi$ for an $n$-representation finite algebra $\Lambda$. 
Since $\Pi_0=\Lambda$ has finite global dimension, the stable category $\underline{\mathrm{mod}}^{\mathbb{Z}}\Pi$ has a tilting object $T$ given in \eqref{df_T}.
Thus we have a triangle-equivalence
\[
\underline{\mathrm{mod}}^{\mathbb{Z}}\Pi \simeq \mathsf{D}^{\mathrm{b}}(\mod \underline{\End}_{\Pi}(T)_0).
\]

The following main result in this section gives a simple description of $\underline{\End}_{\Pi}(T)_0$.

\begin{thm}\label{IOProp2}
Let $\Lambda$ be an $n$-representation finite algebra, $T$ the tilting object in $\underline{\mathrm{mod}}^{\mathbb{Z}}\Pi$ given \eqref{df_T}.
Then there exists an algebra isomorphism 
\[
\underline{\End}_{\Pi}(T)_0 \simeq \underline{\End}_{\Lambda}(\Pi).
\]
\end{thm}

Applying Theorem \ref{equ1}, we immediately obtain the following result given in \cite[Theorem 4.7]{IO2}.

\begin{cor} \cite[Theorem 4.7]{IO2}
Let $\Lambda$ be an $n$-representation finite algebra.
Then there exists a triangle-equivalence
\[
\underline{\mathrm{mod}}^{\mathbb{Z}}\Pi \simeq \mathsf{D}^{\mathrm{b}}(\mod \underline{\End}_{\Lambda}(\Pi)).
\]
\end{cor}

In the rest of this section, we prove Theorem \ref{IOProp2}.
First we recall basic results in higher Auslander-Reiten theory.
A key role is played by the following functors.

\begin{df}\cite{Iya1}
Let $\Lambda$ be an algebra with global dimension at most $n$.
Then the functors
\[
\tau_n:=D\Ext^n_{\Lambda}(-,\Lambda): \mathrm{mod}\Lambda \longrightarrow \mathrm{mod}\Lambda
\]
and
\[
\tau^{-}_n:=\Ext^n_{\Lambda}(D\Lambda,-): \mathrm{mod}\Lambda \longrightarrow \mathrm{mod}\Lambda
\]
are called the \emph{$n$-Auslander-Reiten translations}.
\end{df}

The restrictions of $n$-Auslander-Reiten translations to $n$-cluster tilting modules give a higher analogue of Auslander-Reiten translations.

\begin{prop}\cite[Theorem 2.3]{Iya1} \cite[Proposition 1.3, Theorem 1.6]{Iya2}\label{app_prop1}
Let $\Lambda$ be an $n$-representation finite algebra, and $M$ an $n$-cluster tilting $\Lambda$-module. 
We denote by $M_P$ (respectively, $M_I$) a maximal projective (respectively, injective) direct summand of $M$.
Then the following assertions hold.
\begin{enumerate}
\def\labelenumi{(\theenumi)} 
\item $\tau_n$ induces an equivalence $\add(M/M_P) \rightarrow \add(M/M_I)$, and $\tau^{-}_n$ induces its inverse.
\item For any indecomposable projective $\Lambda$-module $P$, there exists $i \in \mathbb{N}$ such that $\tau_n^{-i}(P)$ is an indecomposable injective $\Lambda$-module.
\item We have $\add M = \add \bigoplus_{i\geq 0}\tau^{-i}_n(\Lambda)$.
\end{enumerate}
\end{prop}

The $(n+1)$-preprojective algebra $\Pi$ of an $n$-representation finite algebra $\Lambda$ can be described in terms of $n$-Auslander-Reiten translation.

\begin{lem} \cite[Lemma 2.13]{IO2} \label{app_lem1}
Let $\Lambda$ be an $n$-representation finite algebra. 
For any $i \in \mathbb{N}$, there is an isomorphism 
\[
\tau_n^{-i}(\Lambda) \simeq \Pi_i
\]
of $\Lambda$-modules. In particular, $\Pi$ is an $n$-cluster tilting $\Lambda$-module (Proposition \ref{app_prop1}).
\end{lem}

Now we are ready to prove Theorem \ref{IOProp2}.

\begin{proof}[Proof of Theorem \ref{IOProp2}]
First we describe $\underline{T}$ in Proposition \ref{cal_end} explicitly.
Let $\mathrm{PI}(\Lambda)$ be a complete set of orthogonal primitive idempotents of $\Lambda$. 
By Proposition \ref{app_prop1} and Lemma \ref{app_lem1}, 
we can decompose $\mathrm{PI}(\Lambda)=\amalg_{i=0}^{\ell}\mathrm{PI}_i(\Lambda)$ where
\[
\mathrm{PI}_i(\Lambda):=\{ e \in \mathrm{PI}(\Lambda) \ | \ \tau^{-i}_n(e\Lambda) \neq 0, \  \tau^{-(i+1)}_n(e\Lambda)=0 \}.
\]
Let 
\[
e_i:=\sum_{e \in \mathrm{PI}_i(\Lambda)}e.
\]
Then by the definition of $\mathrm{PI}_i(\Lambda)$, we have 
\begin{eqnarray}
\tau^{-i}_n(\Lambda)=\tau^{-i}_n((e_i+\cdots+e_{\ell})\Lambda). \label{A}
\end{eqnarray}
We remark that $\mathrm{PI}(\Lambda)$ can be regarded as a complete set of orthogonal primitive idempotents of $\Pi$.

We define $\underline{T}^i$ by putting
\[
\underline{T}^i := ((e_i+e_{i+1}+\cdots+e_{\ell}) \Pi(i-1))_{\leq 0}.
\]
Then $\underline{T}$ defined in Proposition \ref{cal_end} can be described as 
\[
\underline{T}=\bigoplus_{i=1}^{\ell}\underline{T}^i,
\]
and we have $\underline{\End}_{\Pi}(T)_0 \simeq \End_{\Pi}(\underline{T})_0$.

Next we show that $\End_{\Pi}(\underline{T})_0 \simeq \underline{\End}_{\Lambda}(\Pi)$.
By Lemma \ref{app_lem1}, we have
\[
e \Pi =  e\Lambda \oplus \tau^{-}_n(e\Lambda) \oplus \tau^{-2}_n(e\Lambda) \oplus \cdots \oplus \tau^{-i}_n(e\Lambda)
\]
as a $\Lambda$-module for any $e \in \mathrm{PI}_i(\Lambda)$.
Since one can see that
\begin{eqnarray}
(\underline{T}^i)_k = 
\begin{cases}
\tau^{-(k+i-1)}_n((e_i +\cdots+e_{\ell})\Lambda) & (-i+1 \leq k \leq 0)\\
0 & (\mbox{otherwise})
\end{cases}  \label{B}
\end{eqnarray}
for $1 \leq i \leq \ell$, there exist isomorphisms
\begin{eqnarray*}
\Hom_{\Pi}(\underline{T}^i,\underline{T}^j)_0 
&\simeq& \Hom_{\Pi_0}((\underline{T}^i)_{-i+1},(\underline{T}^j)_{-i+1}) \\
& \stackrel{\eqref{B}}{\simeq} & \Hom_{\Lambda}((e_i +\cdots+e_{\ell})\Lambda,\tau_n^{-(j-i)}((e_j +\cdots+e_{\ell})\Lambda)) \\
& \stackrel{\mbox{Prop \ref{app_prop1}}}{\simeq} &
 \Hom_{\Lambda}(\tau_n^{-i}((e_i +\cdots+e_{\ell})\Lambda),\tau_n^{-j}((e_j +\cdots+e_{\ell})\Lambda)) \\
& \stackrel{\eqref{A}}{\simeq}& \Hom_{\Lambda}(\tau_n^{-i}(\Lambda),\tau_n^{-j}(\Lambda))
\end{eqnarray*}
for $i \leq j$, and we have $\Hom_{\Pi}(\underline{T}^i,\underline{T}^j)_0 \simeq \Hom_{\Pi_0}((\underline{T}^i)_{-i+1},(\underline{T}^j)_{-i+1})=0$ for $i > j$.

Thus we have algebra isomorphisms
\begin{eqnarray*}
\End_{\Pi}(\underline{T})_0 & \simeq & \left( \begin{array}{cccc}
\Hom_{\Pi}(\underline{T}^1,\underline{T}^1)_0  & &  & 0 \\
\Hom_{\Pi}(\underline{T}^1,\underline{T}^2)_0  & \Hom_{\Pi}(\underline{T}^2,\underline{T}^2)_0  &  &  \\
\vdots  & \vdots  & \ddots &  \\
\Hom_{\Pi}(\underline{T}^1,\underline{T}^{\ell})_0  & \Hom_{\Pi}(\underline{T}^2,\underline{T}^{\ell})_0  & \cdots & \Hom_{\Pi}(\underline{T}^{\ell},\underline{T}^{\ell})_0 
\end{array} \right) \\ 
& \simeq & \left( \begin{array}{cccc}
\Hom_{\Lambda}(\tau_n^{-1}(\Lambda),\tau_n^{-1}(\Lambda))  &   &  & 0 \\
\Hom_{\Lambda}(\tau_n^{-1}(\Lambda),\tau_n^{-2}(\Lambda))  & \Hom_{\Lambda}(\tau_n^{-2}(\Lambda),\tau_n^{-2}(\Lambda)) & & \\
\vdots  & \vdots & \ddots &  \\
\Hom_{\Lambda}(\tau_n^{-1}(\Lambda),\tau_n^{-\ell}(\Lambda))  & \Hom_{\Lambda}(\tau_n^{-2}(\Lambda),\tau_n^{-\ell}(\Lambda))  & \cdots & \Hom_{\Lambda}(\tau_n^{-\ell}(\Lambda),\tau_n^{-\ell}(\Lambda))
\end{array} \right) \\ 
&=& \End_{\Lambda}\left(\bigoplus_{i \geq 1} \tau_n^{-i}(\Lambda) \right) = \underline{\End}_{\Lambda}(\Pi).
\end{eqnarray*}
\end{proof}

\subsection{Classical cases}
In this subsection, we consider the case $n=1$.
Then $\Lambda$ is the path algebra of a Dynkin quiver $Q$ (see Example \ref{ex_pp}), and $\Pi$ is the classical preprojective algebra associated to the quiver $Q$ (see \cite{BGL,Rin}).
Then $\Pi$ is a positively graded self-injective algebra by Proposition \ref{pp_selfinj}, 
and $\Pi$ is $1$-cluster tilting as a $\Lambda$-module (i.e. $\add \Pi = \mod\Lambda$) by Proposition \ref{app_prop1} and Lemma \ref{app_lem1}. 
By Theorem \ref{IOProp2}, there exists a triangle-equivalence
\[
\underline{\mod}^{\mathbb{Z}}\Pi \simeq \mathsf{D}^{\mathrm{b}}(\mod \underline{\End}_{\Lambda}(\Pi)).
\]

Now we recall a description of $\Pi$ in terms of quivers with relations. 
\begin{list}{$\bullet$}{}
\item Define the double quiver $\overline{Q}$ of $Q$ by
\[
\overline{Q}_0:=Q_0,
\]
\[
\overline{Q}_1:=Q_1 \bigsqcup \left\{  j \xrightarrow{\alpha^*} i \ | \ i \xrightarrow{\alpha} j \in Q_1 \right\}.
\]
Then we have an involution $( \ \ )^*:  \overline{Q}_1\longrightarrow \overline{Q}_1$ defined by
\[
\alpha^* := \begin{cases}
\alpha^* & (\alpha \in Q_1), \\
\beta & (\alpha=\beta^*  \mbox{ for some } \beta \in Q_1).
\end{cases}
\]
\item Define relations $\rho_i$ for each $i \in \overline{Q}_0 $ by
\[
\rho_i := \sum_{\left(i \xrightarrow{\alpha} j\right) \in \overline{Q}_1} \epsilon_{\alpha}  \alpha \alpha^*,
\]
\[
\epsilon_{\alpha} := \begin{cases}
1 & (\alpha \in Q_1), \\
-1 & (\alpha^* \in Q_1).
\end{cases}
\]
\item Define the grading on the quiver $\overline{Q}$ by 
\[
\deg \alpha := \begin{cases}
0 & (\alpha \in Q_1),  \\
1 &  (\alpha^* \in Q_1). \\
\end{cases}
\]
\end{list}

Since the relations $\rho_i$ are homogeneous of degree $1$ with respect to the above grading on $\overline{Q}$, 
the above grading on $\overline{Q}$ gives rise to the grading on $K\overline{Q}/\langle \rho_i \ | \ i \in \overline{Q}_0 \rangle$. 
Then there exists an isomorphism of graded algebras
\[
\Pi \simeq K\overline{Q}/\langle \rho_i \ | \ i \in \overline{Q}_0 \rangle.
\]

Let us give a concrete example. 
Let $Q$ be the quiver 
\[
\xymatrix{
1 \ar[r]^{\alpha_1} & 2 \ar[r]^{\alpha_2} & \cdots \cdots \ar[r]^{\alpha_{n-2}} & n-1 \ar[r]^{\alpha_{n-1}} & n
}
\]
of type $A_n$, and $\Lambda$ the path algebra of $Q$.
Then $\Pi$ and $\underline{\End}_{\Lambda}(\Pi)$ are described by quivers with relations as follows.

First the quiver with relations of $\Pi$ is given by
\[
\xymatrix{
1 \ar@<0.5ex>[r]^{\alpha_1} & 2 \ar@<0.5ex>[r]^{\alpha_2} \ar@<0.5ex>[l]^{\alpha^*_1} & \cdots \cdots \ar@<0.5ex>[r]^{\alpha_{n-2}} \ar@<0.5ex>[l]^{\alpha^*_2} & n-1 \ar@<0.5ex>[r]^{\alpha_{n-1}} \ar@<0.5ex>[l]^{\alpha^*_{n-2}}& n \ar@<0.5ex>[l]^{\alpha^*_{n-1}}
}
\hspace{0.4cm}
\begin{cases}
\alpha_1\alpha^*_1 = 0 =\alpha^*_{n-1}\alpha_{n-1}, & \\
\alpha_i \alpha^*_i = \alpha^*_{i-1}  \alpha_{i-1} & (2 \leq i \leq n-1).
\end{cases}
\]
The grading on $\Pi$ is given by $\deg \alpha_i :=0$ and $\deg \alpha^*_i:=1$.

Next we draw the quiver with relations of $\underline{\End}_{\Lambda}(\Pi)$.
The Auslander-Reiten quiver of $\mod \Lambda=\add \Pi$ is given by
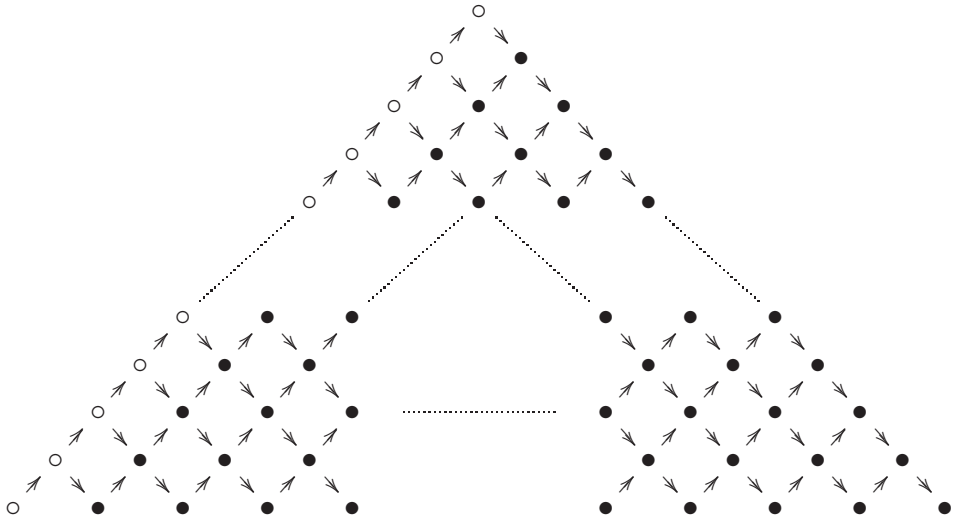
\[
\xymatrix@C0.1cm@R0.2cm{
& & & & & & & & & & & \circ \ar[rd] & & & & & & & & & & & \\
& & & & & & & & & & \circ \ar[rd] \ar[ru]& & \bullet \ar[rd]& & & & & & & & &  & \\
& & & & & & & & & \circ \ar[rd]\ar[ru]& & \bullet \ar[rd] \ar[ru]& & \bullet \ar[rd]& & & & & & & && \\
& & & & & & & & \circ \ar[rd] \ar[ru]& & \bullet \ar[rd]\ar[ru]& & \bullet \ar[rd] \ar[ru]& &\bullet\ar[rd] & & & & & & & &  \\
& & & & & & &  \circ \ar[ru] & & \bullet \ar[ru] & & \bullet \ar[ru] \ar@{.}[rrrddd] & & \bullet \ar[ru]  & & \bullet\ar@{.}[rrrddd]  & & & & & & & \\
& & & & & &  & & & & & & & & & & & & & & & & \\
& & & & &  & & & & & & & & & & & & & & & & &  \\
& &  & & \circ \ar[rd] \ar@{.}[rrruuu] & & \bullet \ar[rd] & & \bullet \ar@{.}[rrruuu] & & & & & & \bullet  \ar[rd] & & \bullet \ar[rd] & & \bullet \ar[rd] & & & & \\
&  & & \circ \ar[ru] \ar[rd] & & \bullet \ar[ru] \ar[rd] & & \bullet \ar[ru]\ar[rd] & & & & & & & & \bullet \ar[ru]\ar[rd] & & \bullet\ar[rd]  \ar[ru]& & \bullet \ar[rd] & & & \\
& & \circ \ar[ru] \ar[rd] & & \bullet \ar[ru] \ar[rd] & & \bullet\ar[ru] \ar[rd] & & \bullet & \ar@{.}[rrrr]& & & & &\bullet\ar[ru] \ar[rd] & &\bullet \ar[ru]\ar[rd] & & \bullet \ar[ru]\ar[rd] & & \bullet \ar[rd] & & \\
& \circ\ar[ru] \ar[rd]  & & \bullet \ar[ru] \ar[rd] & & \bullet \ar[ru] \ar[rd] & & \bullet \ar[ru] \ar[rd]& & & & & & & & \bullet\ar[ru]\ar[rd] & & \bullet\ar[ru]\ar[rd] & & \bullet\ar[ru]\ar[rd] & & \bullet \ar[rd] & \\
\circ \ar[ru] & & \bullet \ar[ru] & & \bullet \ar[ru]  & & \bullet \ar[ru]  & & \bullet & & & & & &\bullet \ar[ru]& &\bullet \ar[ru]& &\bullet \ar[ru]& & \bullet \ar[ru]& & \bullet 
}
\]
where the number of vertices $\circ$ is $n$. 
These vertices $\circ$ correspond to indecomposable projective $\Lambda$-modules.
Thus $\underline{\End}_{\Lambda}(\Pi)$ is given by the full subquiver consisting of vertices $\bullet$ with mesh relations.

\section{Derived-orbit categories}
\label{DG_section}

In this section, we collect basic facts on DG categories, together with an application to orbit categories of derived categories.
In particular, we give the definition of derived-orbit categories and their universal property.

In \cite{Ke1}, B. Keller introduced the derived category of differential graded category (DG category),  and developed fundamental theory. 
It gives a powerful tool for study of algebraic triangulated categories which are often realized as derived categories of DG categories. 

One of their application is to construct derive-orbit categories \cite{Ke2}. 
They are triangulated categories associated with orbit categories of derived categories, which do not have a structure of triangulated categories in general. 
In their construction, we need to introduce DG orbit categories.

\subsection{DG categories}
In this subsection, we recall basic definitions and typical examples of DG categories.
We refer to \cite{Ke1} for the details.

\begin{df}
An additive category $\mathscr{A}$ is called a \emph{differential graded (DG) category} if the following conditions are satisfied.
\begin{enumerate}
\def\labelenumi{(\theenumi)} 
\item For any $X,Y \in \mathscr{A}$, the morphism set $\Hom_{\mathscr{A}}(X,Y)$ is a $\mathbb{Z}$-graded abelian group
\[
\Hom_{\mathscr{A}}(X,Y)= \bigoplus_{i \in \mathbb{Z}}\Hom_{\mathscr{A}}^i(X,Y).
\]
such that $gf \in \Hom_{\mathscr{A}}^{i+j}(X,Z)$ for any $f\in \Hom_{\mathscr{A}}^i(X,Y)$ and $g \in \Hom_{\mathscr{A}}^j(Y,Z)$ where $X,Y,Z \in \mathscr{A}$ and $i,j \in \mathbb{Z}$.
\item The morphism set are endowed with a differential $d$ such that the equation
\[
d(gf)=(dg)f+(-1)^jg(df)
\]
holds for any $f\in \Hom_{\mathscr{A}}(X,Y)$ and $g \in \Hom_{\mathscr{A}}^j(Y,Z)$ where $X,Y,Z \in \mathscr{A}$ and $j \in \mathbb{Z}$.
\end{enumerate}

For a DG category $\mathscr{A}$, we denote by $\mathrm{H}^0(\mathscr{A})$ an additive category whose objects are the same as $\mathscr{A}$, and the morphism set from $X$ to $Y$ is $\mathrm{H}^0(\Hom_{\mathscr{A}}(X,Y))$.

Let $\mathscr{A}$ and $\mathscr{B}$ be DG categories. 
An additive functor $F: \mathscr{A} \rightarrow \mathscr{B}$ is called a \emph{DG functor} if it preserves the grading and commutes with differentials. 
Clearly a DG functor $\mathscr{A} \rightarrow \mathscr{B}$ induces an additive functor $\mathrm{H}^0(\mathscr{A}) \rightarrow \mathrm{H}^0(\mathscr{B})$.
\end{df}

The following are typical examples of DG categories. 
Those shows that usual categories of complexes are dealt with in the context of DG categories.

\begin{ex}\label{ex_DG}
(1) Let $\mathcal{A}$ be an additive category. 
We can regard $\mathcal{A}$ as a DG category such that morphism sets are concentrated in degree $0$, namely $\Hom^0_{\mathcal{A}}(X,Y)=\Hom_{\mathcal{A}}(X,Y)$ for any $X,Y \in \mathcal{A}$.

(2) Let $\mathcal{A}$ be an additive category.
We define the DG category $\mathsf{C}_{\dg}(\mathcal{A})$ as follows. 
\begin{list}{$\bullet$}{}
\item The objects are chain complexes over $\mathcal{A}$. 
\item For $X,Y \in \mathsf{C}_{\dg}(\mathcal{A})$, the morphism sets from $X$ to $Y$ is defined by
\[
\Hom^n_{\mathsf{C}_{\dg}(\mathcal{A})}(X,Y):= \prod_{i \in \mathbb{Z}} \Hom_{\mathcal{A}}(X^i,Y^{i+n}),
\]
\[
\Hom_{\mathsf{C}_{\dg}(\mathcal{A})}(X,Y):=\bigoplus_{n \in \mathbb{Z}}\Hom^n_{\mathsf{C}_{\dg}(\mathcal{A})}(X,Y).
\]
\item The differential is given by
\[
d^n((f_i)_{i \in \mathbb{Z}}) = (d_Y f_i - (-1)^n f_{i+1} d_X)_{i \in \mathbb{Z}}
\]
for $(f_i)_{i \in \mathbb{Z}} \in \Hom^n_{\mathsf{C}_{\dg}(\mathcal{A})}(X,Y)$.
\end{list}

Then one can check that the category $\mathrm{H}^0(\mathsf{C}_{\dg}(\mathcal{A}))$ coincides with the homotopy category $\mathsf{K}(\mathcal{A})$ of complexes over $\mathcal{A}$.

Later we deal with the full DG subcategory $\mathsf{C}_{\dg}^{\mathrm{b}}(\mathcal{A})$ of $\mathsf{C}_{\dg}(\mathcal{A})$ which consists of bounded complexes over $\mathcal{A}$. 
Clearly we have $\mathsf{K}^{\mathrm{b}}(\mathcal{A})=\mathrm{H}^0(\mathsf{C}_{\dg}^{\mathrm{b}}(\mathcal{A}))$.
\end{ex}

Next we recall the derived categories of DG modules over DG categories.
Let $\mathscr{A}$ be a DG category. 
A \emph{DG $\mathscr{A}$-module} is a DG functor $\mathscr{A}^{\op} \rightarrow \mathsf{C}_{\dg}(\Mod \mathbb{Z})$.
Then DG $\mathscr{A}$-modules form a DG category which we denote by $\mathsf{C}_{\dg}(\mathscr{A})$. 
We denote by $\mathsf{K}(\mathscr{A}):=\mathrm{H}^0(\mathsf{C}_{\dg}(\mathscr{A}))$ the homotopy category of DG $\mathscr{A}$-modules.
By formally inverting quasi-isomorphisms in $\mathsf{K}(\mathscr{A})$,  we obtain the derived category $\mathsf{D}(\mathscr{A})$ of DG $\mathscr{A}$-modules.

For a DG category $\mathscr{A}$, the homotopy category $\mathsf{K}(\mathscr{A})$ has a structure of a triangulated category.
Moreover the derived category $\mathsf{D}(\mathscr{A})$ also has a structure of a triangulated category such that the canonical functor $\mathsf{K}(\mathscr{A}) \rightarrow \mathsf{D}(\mathscr{A})$ is a triangle-functor.

The above derived categories contains the derived categories of algebras.

\begin{ex}\label{ex_DG_alg}
Let $\Lambda$ be an algebra. 
We denote by $\Mod\Lambda$ the category of all $\Lambda$-modules.
As we observed in Example \ref{ex_DG} (1), we regard $\Lambda$ as a DG category whose morphism space is concentrated in degree $0$.
So we can consider the DG category $\mathsf{C}_{\dg}(\Lambda)$ of DG $\Lambda$-modules introduced in the above.
This coincides with the DG category $\mathsf{C}_{\dg}(\Mod \Lambda)$ of complexes over $\Mod\Lambda$ discussed in Example \ref{ex_DG} (2). 
Also we have $\mathsf{K}(\Lambda)=\mathsf{K}(\Mod \Lambda)$ and $\mathsf{D}(\Lambda)=\mathsf{D}(\Mod \Lambda)$.
\end{ex}

For a DG category $\mathscr{A}$, the category $\mathrm{H}^0(\mathscr{A})$ does not have a structure of a triangulated category in general. 
However we can naturally construct a triangulated category which contains $\mathrm{H}^0(\mathscr{A})$ as a full subcategory.

We call a DG functor $\mathscr{A} \rightarrow \mathsf{C}_{\dg}(\mathscr{A})$ sending $X$ to $\Hom_{\mathscr{A}}(-,X)$ the \emph{Yoneda embedding}. This is fully faithful, and we regard $\mathscr{A}$ as a DG full subcategory of $\mathsf{C}_{\dg}(\mathscr{A})$, 
and $\mathrm{H}^0(\mathscr{A})$ as a full subcategory of $\mathrm{H}^0(\mathsf{C}_{\dg}(\mathscr{A}))=\mathsf{K}(\mathscr{A})$.
Furthermore we regard $\mathrm{H}^0(\mathscr{A})$ as a full subcategory of $\mathsf{D}(\mathscr{A})$ via the composition 
$\mathrm{H}^0(\mathscr{A}) \rightarrow \mathsf{K}(\mathscr{A}) \rightarrow \mathsf{D}(\mathscr{A})$.

\begin{df}\label{df_tri_hull}
Let $\mathscr{A}$ be a DG category. 
We call the full subcategory  $\thick_{\mathsf{D}(\mathscr{A})}(\mathrm{H}^0(\mathscr{A}))$ of $\mathsf{D}(\mathscr{A})$ the \emph{triangulated hull} of $\mathrm{H}^0(\mathscr{A})$. 
\end{df}

This is a triangulated category which contains $\mathrm{H}^0(\mathscr{A})$ as a full subcategory and is generated by $\mathrm{H}^0(\mathscr{A})$.

\subsection{Derived-orbit categories}

The aim of this subsection is to introduce derived-orbit categories following \cite{Ke2}.
First we recall DG orbit categories.

\begin{df}
Let $\mathscr{A}$ be a DG category, and $F:\mathscr{A} \rightarrow \mathscr{A}$ a DG functor.
A DG  category $\mathscr{A}/F^{+}$ is defined as follows. 
\begin{list}{$\bullet$}{}
\item The objects are the same as $\mathscr{A}$.
\item For $X, Y \in \mathscr{A}$, the morphism set from $X$ to $Y$ is defined by
\[
\Hom_{\mathscr{A}/F^{+}}(X,Y):=\bigoplus_{n=0}^{\infty} \Hom_{\mathscr{A}}(F^nX,Y).
\]
\item The composition of morphisms  are defined as follows. 
For $f \in \Hom_{\mathscr{A}}(F^iX,Y) \subset \Hom_{\mathscr{A}/F^{+}}(X,Y)$ and $g \in \Hom_{\mathscr{A}}(F^jY,Z) \subset \Hom_{\mathscr{A}/F^{+}}(Y,Z)$, we define
\[
g\cdot  f := g \cdot F^j(f)  \in \Hom_{\mathscr{A}}(F^{i+j}X,Z) \subset \Hom_{\mathscr{A}/F^{+}}(X,Z).
\]
\item The differential of $\mathscr{A}/F^{+}$ is induced from that of $\mathscr{A}$.
\end{list}
\end{df}

\begin{df}
Let $\mathscr{A}$ be a DG category, and $F:\mathscr{A} \rightarrow \mathscr{A}$ a DG functor.
The \emph{DG orbit category} $\mathscr{A}/F$ is defined as follows. 
\begin{list}{$\bullet$}{}
\item The objects are the same as $\mathscr{A}$.
\item For $X, Y \in \mathscr{A}$, the morphism set from $X$ to $Y$ is defined by
\[
\Hom_{\mathscr{A}/F}(X,Y):= \mathrm{colim} \left( \Hom_{\mathscr{A}/F^{+}}(X,Y) \xrightarrow{F} \Hom_{\mathscr{A}/F^{+}}(X,FY) \xrightarrow{F} \cdots \cdots \right).
\]
\item The composition of morphisms in $\mathscr{A}/F$ is induced from that in $\mathscr{A}/F^{+}$.
\item The differential of $\mathscr{A}/F$ is induced from that of $\mathscr{A}/F^{+}$.
\end{list}
\end{df}

\begin{ex}\label{df_orbit}
As we observed in Example \ref{ex_DG} (1), an additive category $\mathcal{A}$ can be regarded as a DG category. 
For an additive functor $F: \mathcal{A} \rightarrow \mathcal{A}$, we define the orbit category $\mathcal{A}/F$ as a special case of the above definition.
\end{ex}

%

The following proposition explains the reason why we consider DG orbit categories. 
It allows us to deal with orbit categories in the context of DG orbit categories.

\begin{prop} \label{App_1}
Let $\mathscr{A}$ be a DG category, $F: \mathscr{A} \rightarrow \mathscr{A}$ a DG functor, and the induced functor $\mathrm{H}^0(F):\mathrm{H}^0(\mathscr{A}) \rightarrow \mathrm{H}^0(\mathscr{A})$. 
Then the following assertions hold.
\begin{enumerate}
\def\labelenumi{(\theenumi)} 
\item There exists an equivalence $\mathrm{H}^0(\mathscr{A}/F) \simeq \mathrm{H}^0(\mathscr{A})/\mathrm{H}^0(F)$. 
\item The triangulated hull $\thick_{\mathsf{D}(\mathscr{A}/F)} (\mathrm{H}^0(\mathscr{A}/F))$ of $\mathrm{H}^0(\mathscr{A}/F)$ is a triangulated category which contains $\mathrm{H}^0(\mathscr{A})/\mathrm{H}^0(F)$ as a full subcategory, and is generated by $\mathrm{H}^0(\mathscr{A})/\mathrm{H}^0(F)$. 
\end{enumerate}
\end{prop}
\begin{proof}
(1) Clearly we have $\mathrm{H}^0(\mathscr{A}/F^{+}) \simeq \mathrm{H}^0(\mathscr{A})/\mathrm{H}^0(F)^{+}$. 
For any $X,Y \in \mathscr{A}$, there are isomorphisms 
\begin{eqnarray*}
& & \Hom_{\mathrm{H}^0(\mathscr{A}/F)}(X,Y) \\
&\simeq& \mathrm{H}^0\left(\mathrm{colim} \left( \Hom_{\mathscr{A}/F^{+}}(X,Y) \xrightarrow{F} \Hom_{\mathscr{A}/F^{+}}(X,FY) \xrightarrow{F} \cdots \cdots \right)\right) \\
&\simeq& \mathrm{colim} \left(\Hom_{ \mathrm{H}^0(\mathscr{A}/F^{+})}(X,Y) \xrightarrow{\mathrm{H}^0(F)} \Hom_{\mathrm{H}^0(\mathscr{A}/F^{+})}(X,FY) \xrightarrow{\mathrm{H}^0(F)} \cdots \cdots \right) \\ 
&\simeq& \mathrm{colim} \left( \Hom_{\mathrm{H}^0(\mathscr{A})/\mathrm{H}^0(F)^{+}}(X,Y) \xrightarrow{\mathrm{H}^0(F)} \Hom_{\mathrm{H}^0(\mathscr{A})/\mathrm{H}^0(F)^{+}}(X,FY) \xrightarrow{\mathrm{H}^0(F)} \cdots \cdots \right) \\
&\simeq& \Hom_{\mathrm{H}^0(\mathscr{A})/\mathrm{H}^0(F)}(X,Y).
\end{eqnarray*}
The second isomorphism follows from exactness of colimit (cf. \cite[Theorem III.1.9]{Mi}).

(2) This is clear. 
\end{proof}

%
%
%

Now we are ready to define derived-orbit categories.

\begin{df}\label{df_orbit}
Let $\Lambda$ be an algebra of finite global dimension, and $M$ a bounded complex of $\Lambda^{\op} \otimes_K \Lambda$-modules.
Let $\mathscr{A}:=\mathsf{C}^{\mathrm{b}}_{\dg}(\proj \Lambda)$ be the DG category of bounded complexes over $\proj\Lambda$ (Example \ref{ex_DG} (2)).
Then we have $\mathrm{H}^0(\mathscr{A})=\mathsf{K}^{\mathrm{b}}(\proj\Lambda)=\mathsf{D}^{\mathrm{b}}(\mod\Lambda)$ by Proposition \ref{Serre=gldim}.
Since $\Lambda^{\op} \otimes_K \Lambda$ has finite global dimension, there exists a quasi-isomorphism $pM \rightarrow M$ with a bounded complex $pM$ of projective $\Lambda^{\op} \otimes_K \Lambda$-modules. 
Consider the DG functor.
\[
F:=-\otimes_{\Lambda} pM:\mathscr{A} \rightarrow \mathscr{A}.
\] 
Then we have 
\[
\mathrm{H}^0(F)=- \Lten_{\Lambda}M : \mathsf{D}^{\mathrm{b}}(\mod\Lambda) \rightarrow \mathsf{D}^{\mathrm{b}}(\mod\Lambda).
\]
We call the triangulated hull
\[
\mathsf{D}(\Lambda,M):=\thick_{\mathsf{D}(\mathscr{A}/F)}(\mathrm{H}^0(\mathscr{A}/F))
\]
of $\mathrm{H}^0(\mathscr{A}/F)$ the \emph{derived-orbit category} of $(\Lambda,M)$.
\end{df}

By Proposition \ref{App_1}, we have $\mathrm{H}^0(\mathscr{A}/F) \simeq \mathrm{H}^0(\mathscr{A})/\mathrm{H}^0(F) \simeq \mathsf{D}^{\mathrm{b}}(\mod \Lambda)/(-\Lten_{\Lambda}M)$.
Thus we have the following result.

\begin{prop}\label{gen_orbit}
The derived-orbit category $\mathsf{D}(\Lambda,M)$ is a triangulated category which contains $\mathsf{D}^{\mathrm{b}}(\mod \Lambda)/(-\Lten_{\Lambda}M)$ as a full subcategory. 
Moreover $\mathsf{D}^{\mathrm{b}}(\mod \Lambda)/(-\Lten_{\Lambda}M)$ is closed under shifts $[\pm 1]$ and generates $\mathsf{D}(\Gamma,M[1])$ as a triangulated category.
\end{prop}

\subsection{A universal property of derived-orbit categories}
In this subsection, we recall a universal property of derived-orbit categories following \cite[Section 9]{Ke2} and \cite[Theorem A. 20]{IO2}.

Let $A$ be a positively graded self-injective algebra, and $\Gamma$ an algebra of finite global dimension. 
Let $U \in \mathsf{D}^{\mathrm{b}}(\mod^{\mathbb{Z}} (\Gamma^{\op} \otimes_K A))$, $N \in \mathsf{D}^{\mathrm{b}}(\mod( \Gamma^{\op} \otimes_K \Gamma))$, and
\[
F:=-\Lten_{\Gamma}N: \mathsf{D}^{\mathrm{b}}(\mod\Gamma) \longrightarrow \mathsf{D}^{\mathrm{b}}(\mod\Gamma).
\]
We consider the derived tensor functor
\[
-\Lten_{\Gamma}U: \mathsf{D}^{\mathrm{b}}(\mod\Gamma) \longrightarrow \mathsf{D}^{\mathrm{b}}(\mod^{\mathbb{Z}}A)
\]
and the canonical functor 
\[
H: \mathsf{D}^{\mathrm{b}}(\mod^{\mathbb{Z}}A) \longrightarrow \underline{\mod}^{\mathbb{Z}}A
\]
given in \eqref{df_H}.
Composing them, we have a triangle-functor
\begin{eqnarray}
G:=H \circ (-\Lten_{\Gamma}U): \mathsf{D}^{\mathrm{b}}(\mod\Gamma) \longrightarrow \underline{\mod}^{\mathbb{Z}}A. \label{comp_fun}
\end{eqnarray}

Now we can state a universal property of derived-orbit categories.

\begin{thm}\label{univ_orbit}\cite[Theorem A. 20]{IO2}
We fix an integer $a$. 
Assume that there exists a triangle
\begin{eqnarray}
P \longrightarrow U(a) \longrightarrow N \Lten_{\Gamma} U \longrightarrow P[1]  \label{univ_tri}
\end{eqnarray}
in $\mathsf{D}^{\mathrm{b}}(\mod^{\mathbb{Z}} (\Gamma^{\op} \otimes_K A))$ such that 
$P$ belongs to $\mathsf{K}^{\mathrm{b}}(\proj^{\mathbb{Z}}A)$ as an object in $\mathsf{D}^{\mathrm{b}}(\mod^{\mathbb{Z}}A)$.

Then there exist an additive functor $\mathsf{D}^{\mathrm{b}}(\mod \Gamma)/F \rightarrow (\underline{\mod}^{\mathbb{Z}}A)/(a)$ and a triangle-functor $\widetilde{G}:\mathsf{D}(\Gamma,N) \rightarrow \underline{\mod}^{\mathbb{Z}/a\mathbb{Z}}A$ which makes the diagram
\[
\xymatrix{
\mathsf{D}^{\mathrm{b}}(\mod \Gamma) \ar[rr]^{G}  \ar[d]_{\mathrm{nat.}} & & \underline{\mod}^{\mathbb{Z}}A \ar[d]^{\mathrm{nat.}} \ar@/^20mm/[dd]^{\underline{F_a}}\\
\mathsf{D}^{\mathrm{b}}(\mod \Gamma)/F \ar[rr]  \ar[d] & & (\underline{\mod}^{\mathbb{Z}}A)/(a) \ar[d] \\
\mathsf{D}(\Gamma,N) \ar[rr]_{\widetilde{G}} & & \underline{\mod}^{\mathbb{Z}/a\mathbb{Z}}A
}
\]
commutative. 
\end{thm}

\section{Realizing stable categories as derived-orbit categories}
\label{Gor_section}

In this section, we compare the stable categories of self-injective algebras and derived-orbit categories. 
We give a precise statement of Theorem \ref{intro_thm3} by giving $\Gamma$ and $M$ explicitly and prove it.
We start with the following motivating observation.

Let $A$ be a positively graded self-injective algebra and $\ell$ a fixed integer.
By Proposition \ref{Z_to_Z/aZ}, the triangle-functor $\underline{F_{\ell}}: \underline{\mod}^{\mathbb{Z}}A \rightarrow \underline{\mod}^{\mathbb{Z}/\ell\mathbb{Z}}A$ induces a fully faithful functor $(\underline{\mod}^{\mathbb{Z}}A)/(\ell) \rightarrow \underline{\mod}^{\mathbb{Z}/\ell\mathbb{Z}}A$ where $(\underline{\mod}^{\mathbb{Z}}A)/(\ell)$ is the orbit category of $\underline{\mod}^{\mathbb{Z}}A$ with respect to $(\ell)$. 
Moreover $\thick_{\underline{\mod}^{\mathbb{Z}/\ell\mathbb{Z}}A}((\underline{\mod}^{\mathbb{Z}}A)/(\ell)) =\underline{\mod}^{\mathbb{Z}/\ell\mathbb{Z}}A$ holds by Proposition \ref{sma_tri}.

In view of the equivalence $\underline{\mod}^{\mathbb{Z}}A \simeq \mathsf{D}^{\mathrm{b}}(\mod\Gamma)$ in Theorem \ref{equ1}, it is reasonable to expect that $\underline{\mod}^{\mathbb{Z}/\ell\mathbb{Z}}A$ is realized as a derived-orbit category $\mathsf{D}(\Gamma,M)$ for a certain $M$.
In fact, we see in Theorem \ref{hull_vs_graded} that this is the case if $A$ has Gorensten parameter $\ell$.

\subsection{Main result}

In this subsection, we give a precise statement of Theorem \ref{intro_thm3} in Theorem \ref{hull_vs_graded}.
We fix the following notations.

Let $A$ be a positively graded self-injective algebra of Gorenstein parameter $\ell$ (see Definition \ref{df_Gor_par}) such that $A_0$ has finite global dimension.
Let $T$ be the tilting object in $\underline{\mod}^{\mathbb{Z}}A$ given in $\eqref{df_T}$, 
and $\underline{T}$ the direct summand of $T$ given in Proposition \ref{cal_end}.
Then we have $\underline{T}=\bigoplus_{i=0}^{\ell-1}A(i)_{\leq 0}$. 
We naturally identify relevant algebras
\begin{eqnarray}
\Gamma:=\underline{\End}_{A}(T)_0=\End_A(\underline{T})_0 \stackrel{\mbox{Lem \ref{nonstab_end}}}{=} \left( \begin{array}{ccccc} 
A_0 & A_1 & \cdots & A_{\ell-2} &  A_{\ell-1}  \\
 & A_0 & \cdots & A_{\ell-3} & A_{\ell-2}  \\
 &  & \ddots & \vdots & \vdots \\
 &  &  & A_0 & A_1   \\
0 &  & & &  A_0 
\end{array} \right). \label{df_Beilinson}
\end{eqnarray}

We regard $\Gamma^{\op}\otimes_KA$ as a $\mathbb{Z}$-graded algebra by Definition \ref{tesor_grading}, and regard $\underline{T}$ as a $\mathbb{Z}$-graded $\Gamma^{\op}\otimes_KA$-module.
It is convenient to write $\underline{T}$ as the following matrix form.
\[
\underline{T}=\left( \begin{array}{c}
A(\ell-1)_{\leq 0} \\
A(\ell-2)_{\leq 0} \\
\vdots \\
A(0)_{\leq 0}
\end{array} \right) =
\stackrel{1-\ell \ \ \  2-\ell  \ \ \ \   \cdots \ \ \ \  -1 \ \ \ \ \ \  0 \ \ \ \ \ }{\left( \begin{array}{ccccc} 
A_0 & A_1 & \cdots & A_{\ell-2} &  A_{\ell-1}  \\
 & A_0 & \cdots & A_{\ell-3} & A_{\ell-2}  \\
 &  & \ddots & \vdots & \vdots \\
 &  &  & A_0 & A_1   \\
0 &  & & &  A_0 
\end{array} \right)}.
\]
Here the numbers $1-\ell, 2-\ell, \cdots, -1,0$ above show the degrees.
The action of $\Gamma$ on $\underline{T}$ from the left is given by the matrix multiplication.
Thus $\underline{T}$ is isomorphic to $\Gamma$ as a $\Gamma^{\op}$-module.

Now we consider a $\mathbb{Z}$-graded $A$-module
\[
M:=\bigoplus_{i=\ell}^{2\ell-1}A(i)_{\geq 1-\ell}.
\]
It is convenient to write $M$ as the following matrix form by the similar way to $\underline{T}$.
\begin{eqnarray}
M=\left( \begin{array}{c}
A(2\ell-1)_{\geq 1-\ell} \\
A(2\ell-2)_{\geq 1-\ell} \\
\vdots \\
A(\ell)_{\geq 1-\ell}
\end{array} \right) =\stackrel{1-\ell \ \ \ \  2-\ell  \ \ \ \  \cdots \ \ \  \  -1 \ \ \ \ \ \  0   }{\left( \begin{array}{ccccc} 
A_{\ell} & & &  & 0 \\
A_{\ell-1} & A_{\ell} &  & &   \\
\vdots & \vdots & \ddots &  &  \\
A_2 & A_3 & \cdots & A_{\ell} &  \\
A_1 & A_2  & \cdots & A_{\ell-1} &  A_{\ell} 
\end{array} \right)}. \label{df_M}
\end{eqnarray}

The algebra $\Gamma$ acts on $M$ from both sides by the matrix multiplication, and we can regard $M$ as a $\Gamma^{\op} \otimes_K \Gamma$-module.
Since the action of $\Gamma$ on $M$ from the left commutes with that of $A$ from the right, we can regard $M$ as a $\mathbb{Z}$-graded $\Gamma^{\op} \otimes_K A$-module.

We consider triangle-functors
\[
F:=-\Lten_{\Gamma}M[1]: \mathsf{D}^{\mathrm{b}}(\mod \Gamma) \longrightarrow \mathsf{D}^{\mathrm{b}}(\mod \Gamma),
\]
\[
 G:=H \circ (-\Lten_{\Gamma}\underline{T}): \mathsf{D}^{\mathrm{b}}(\mod \Gamma) \longrightarrow \underline{\mod}^{\mathbb{Z}}A
\]
where $H$ is given in \eqref{df_H}. 
We know by Theorem \ref{equ2} that $G$ is an equivalence. 

We are ready to state our main results in this section.

\begin{prop}\label{req_comm}
We have the following commutative diagram of triangle-equivalences up to an isomorphism of functors.
\begin{eqnarray}
\xymatrix{
\mathsf{D}^{\mathrm{b}}(\mod\Gamma) \ar[d]_{F}  \ar[rr]^G & & \underline{\mod}^{\mathbb{Z}}A \ar[d]^{(\ell)}\\
\mathsf{D}^{\mathrm{b}}(\mod\Gamma) \ar[rr]_G & & \underline{\mod}^{\mathbb{Z}}A
} \label{commu1}
\end{eqnarray}
\end{prop}

The above commutative diagram suggests that the derived-orbit category $\mathsf{D}(\Gamma,M[1])$ is equivalent to $\underline{\mod}^{\mathbb{Z}/\ell\mathbb{Z}}A$, and this is in fact the case by the following result.

\begin{thm}\label{hull_vs_graded}
There exists a triangle-equivalence 
\[
\widetilde{G}:\mathsf{D}(\Gamma,M[1]) \longrightarrow \underline{\mod}^{\mathbb{Z}/\ell\mathbb{Z}}A
\]
which makes the following diagram commutative.
\begin{eqnarray}
\xymatrix{
\mathsf{D}^{\mathrm{b}}(\mod\Gamma) \ar[rr]^G \ar[d]_{\mathrm{nat.}} & & \underline{\mod}^{\mathbb{Z}}A \ar[d]^{\underline{F_{\ell}}} \\
\mathsf{D}(\Gamma,M[1]) \ar[rr]_{\widetilde{G}} & & \underline{\mod}^{\mathbb{Z}/\ell\mathbb{Z}}A 
} \label{commu2}
\end{eqnarray}
\end{thm}

We prove these results in the next subsection.

\subsection{Proof of Theorem \ref{hull_vs_graded}}
We keep the notations in the previous subsection.
In this subsection, we prove Proposition \ref{req_comm} and Theorem \ref{hull_vs_graded}. 
Our strategy is to construct a triangle \eqref{univ_tri} and apply Theorem \ref{univ_orbit}. 

We start with the following observation.

\begin{lem}\label{ten_lem}
There exists an isomorphism $M \stackrel{\mathbb{L}}{\otimes}_{\Gamma} \underline{T} \simeq M$ in $\mathsf{D}^{\mathrm{b}}(\mod^{\mathbb{Z}}(\Gamma^{\op} \otimes_K A))$.
\end{lem}
\begin{proof}
First since $\underline{T}$ is isomorphic to $\Gamma$  as a $\Gamma^{\op}$-module, we have $M \stackrel{\mathbb{L}}{\otimes}_{\Gamma} \underline{T} = M \otimes_{\Gamma} \underline{T}$.
Next the matrix multiplication gives an isomorphism 
\[
f:M \otimes_{\Gamma} \underline{T} \ni m \otimes x \longmapsto mx \in M
\]
of $(\Gamma^{\op}\otimes_KA)$-modules.
Finally since $f(M \otimes_{\Gamma}(\underline{T}_i)) \subset M_i$ holds for any $i \in \mathbb{Z}$, we have that $f$ is an isomorphism in $\mod^{\mathbb{Z}}(\Gamma^{\op} \otimes_K A)$.
\end{proof}

\begin{lem}\label{reqtri_lem}
The following assertions hold.
\begin{enumerate}
\def\labelenumi{(\theenumi)} 
\item There exists an exact sequence
\[
0 \longrightarrow M \longrightarrow \bigoplus_{i=\ell}^{2\ell-1}A(i) \longrightarrow \underline{T}(\ell) \rightarrow 0
\]
in $\mod^{\mathbb{Z}} (\Gamma^{\op} \otimes A)$.
\item There exists a triangle
\begin{eqnarray}
\bigoplus_{i=\ell}^{2\ell-1}A(i) \longrightarrow \underline{T}(\ell) \xrightarrow{\ \phi \ } M[1]  \stackrel{\mathbb{L}}{\otimes}_{\Gamma}\underline{T} \longrightarrow \bigoplus_{i=\ell}^{2\ell-1}A(i)[1]  \label{reqtri}
\end{eqnarray}
in $\mathsf{D}^{\mathrm{b}}(\mod^{\mathbb{Z}}(\Gamma^{\op} \otimes_K A))$.
\end{enumerate}
\end{lem}
\begin{proof}
(1) The following figure shows that the $\mathbb{Z}$-graded $A$-module $M$ is the first syzygy of the $\mathbb{Z}$-graded $A$-module $\underline{T}(\ell)$.
\[
\left( \begin{array}{c|c} \underline{T}(\ell) & M \end{array} \right)=
\left(\begin{array}{ccccc|cccccc}
A_0 & A_1 & \cdots & A_{\ell-2} &  A_{\ell-1} & A_{\ell} & & &  & 0 \\
 & A_0 & \cdots & A_{\ell-3} & A_{\ell-2}  & A_{\ell-1} & A_{\ell} &  & &   \\
 &  & \ddots & \vdots & \vdots & \vdots & \vdots & \ddots &  &  \\
 &  &  & A_0 & A_1  & A_2 & A_3 & \cdots & A_{\ell} &  \\
0 &  & & &  A_0 & A_1 & A_2  & \cdots & A_{\ell-1} &  A_{\ell} 
\end{array} \right)
\]
Thus there is an exact sequence
\[
0 \longrightarrow M \longrightarrow \bigoplus_{i=\ell}^{2\ell-1}A(i) \longrightarrow \underline{T}(\ell) \rightarrow 0
\]
in $\mod^{\mathbb{Z}}A$.
We can show $\End_A(\bigoplus_{i=\ell}^{2\ell-1}A(i))_0 \simeq \Gamma$ by the similar way as in Lemma \ref{nonstab_end}. 
Thus $\bigoplus_{i=\ell}^{2\ell-1}A(i)$ is a $\mathbb{Z}$-graded $\Gamma^{\op} \otimes_KA$-module. 
It is easy to check that the above exact sequence is in $\mod^{\mathbb{Z}}(\Gamma^{\op}\otimes_KA)$.
Thus we have the assertion.

(2) Since each short exact sequence in $\mod^{\mathbb{Z}}(\Gamma^{\op} \otimes A)$ gives a triangle in $\mathsf{D}(\mod^{\mathbb{Z}}(\Gamma^{\op} \otimes_K A))$, the assertion follows from (1) and Lemma \ref{ten_lem}.
\end{proof}

Now we are ready to prove Proposition \ref{req_comm}  and Theorem \ref{hull_vs_graded}.

\begin{proof}[Proof of Proposition \ref{req_comm}]
We use the map $\phi$ given in \eqref{reqtri}. Let $P=\bigoplus_{i=\ell}^{2\ell-1}A(i)$.
For $X \in \mathsf{D}^{\mathrm{b}}(\mod\Gamma)$, we have a triangle 
\[
H(X \Lten_{\Gamma} P) \longrightarrow H(X \Lten_{\Gamma}\underline{T}(\ell)) \xrightarrow{H(X \Lten_{\Gamma} \phi) } H(X \Lten_{\Gamma}M[1]  \stackrel{\mathbb{L}}{\otimes}_{\Gamma}\underline{T}) \longrightarrow H(X \Lten_{\Gamma}P)[1]
\]
in $\underline{\mod}^{\mathbb{Z}}A$ by applying $H(X \Lten_{\Gamma}-)$ to the triangle \eqref{reqtri}.  
Here we have $H(X \Lten_{\Gamma}\underline{T}(\ell))=G(X)(\ell)$ and $H(X \Lten_{\Gamma}M[1]  \stackrel{\mathbb{L}}{\otimes}_{\Gamma}\underline{T})=GF(X)$.
Since $X \Lten_{\Gamma} P$ belongs to $\mathsf{K}^{\mathrm{b}}(\proj^{\mathbb{Z}}A)$ as an object in $\mathsf{D}^{\mathrm{b}}(\mod^{\mathbb{Z}}A)$, we have $H(X \Lten_{\Gamma} P)=0$.
Thus we have an isomorphism $H(X \Lten_{\Gamma} \phi):G(X)(\ell) \rightarrow GF(X)$ in $\underline{\mod}^{\mathbb{Z}}A$.
Consequently we have an isomorphism $H(- \Lten_{\Gamma} \phi) : (\ell) \circ G \longrightarrow G \circ F$ of functors.
\end{proof}

Thanks to Proposition \ref{req_comm}, we can show the following property of the derived-orbit category $\mathsf{D}(\Gamma,M[1])$.

\begin{prop}\label{der-orb_KS}
The following assertions hold.
\begin{enumerate}
\def\labelenumi{(\theenumi)} 
\item We have an equivalence $\mathsf{D}^{\mathrm{b}}(\mod\Gamma)/F \simeq (\underline{\mod}^{\mathbb{Z}}A)/(\ell)$.
\item The derived-orbit category $\mathsf{D}(\Gamma,M[1])$ is a Hom-finite Krull-Schmidt category.
\end{enumerate}
\end{prop}
\begin{proof}
(1) By Proposition \ref{req_comm}, $G$ induces an equivalence $\mathsf{D}^{\mathrm{b}}(\mod \Gamma)/F \simeq (\underline{\mod}^{\mathbb{Z}}A)/(\ell)$.

(2) By (1), $\mathsf{D}^{\mathrm{b}}(\mod \Gamma)/F$ is Hom-finite.
By this fact and since $\mathsf{D}^{\mathrm{b}}(\mod \Gamma)/F$ is closed under shifts $[\pm 1]$ and generates $\mathsf{D}(\Gamma,M[1])$ as a triangulated category, $\mathsf{D}(\Gamma,M[1])$ is Hom-finite by Lemma \ref{uses_five_Hom}.

To show that $\mathsf{D}(\Gamma,M[1])$ is Krull-Schmidt, we only have to show that it is idempotent split. 
The category $\mathsf{D}(\Gamma,M[1])$ is defined as a thick subcategory of the derived category of some DG orbit category (see Definition \ref{df_orbit}). 
By \cite[Proposition 1.6.8]{Ne}, the derived categories of DG categories are idempotent split.  
Thus $\mathsf{D}(\Gamma,M[1])$ is idempotent split.
\end{proof}

\begin{proof}[Proof of Theorem \ref{hull_vs_graded}]
Applying Theorem \ref{univ_orbit} to the triangle \eqref{reqtri} given in Lemma \ref{reqtri_lem} (2), there exists a triangle-functor $\widetilde{G}:\mathsf{D}(\Gamma,M[1]) \rightarrow \underline{\mod}^{\mathbb{Z}/\ell\mathbb{Z}}A$ which makes the diagram \eqref{commu2} commutative.

We show that $\widetilde{G}$ is an equivalence. 
By Theorem \ref{univ_orbit} and Proposition \ref{der-orb_KS} (1), we have a commutative diagram 
\[
\xymatrix{
\mathsf{D}^{\mathrm{b}}(\mod \Gamma)/F  \ar[rr]  \ar[d] & &  (\underline{\mod}^{\mathbb{Z}}A)/(\ell) \ar[d] \\
\mathsf{D}(\Gamma,M[1])  \ar[rr]_{\widetilde{G}} & & \underline{\mod}^{\mathbb{Z}/\ell\mathbb{Z}}A
}
\]
where the upper functor is an equivalence. 
Applying Lemma \ref{uses_five}, we have that $\widetilde{G}$ is an equivalence.
\end{proof}

\subsection{Examples}
In this subsection, we give examples of Theorem \ref{hull_vs_graded}.

First we consider positively graded symmetric algbras.

\begin{ex}\label{sym_hull_vs_graded}
Let $A$ be a positively graded symmetric algebra of Gorenstein parameter $\ell$, $\Gamma$ the algebra given in \eqref{df_Beilinson}, and $M$ the $\Gamma^{\op} \otimes_K \Gamma$-module given in \eqref{df_M}.

Now we show that there exists an isomorphism $M \simeq D\Gamma$ of $\Gamma ^{\op}\otimes_K \Gamma$-modules.
Indeed since $A$ is symmetric and has Gorenstein parameter $\ell$, there is an isomorphism
\[
\varphi: A \longrightarrow D(A(\ell)) 
\]
of graded $A^{\op}\otimes_KA$-modules.
Then the restriction of $\varphi$ gives an isomorphism 
\[
\varphi |_{A_i}: A_i \longrightarrow D(A_{\ell-i})
\]
of $K$-vector spaces.

Now we define an isomorphism 
\[
\widetilde{\varphi}: M=\left( \begin{array}{ccccc} 
A_{\ell} & & &  & 0 \\
A_{\ell-1} & A_{\ell} &  & &   \\
\vdots & \vdots & \ddots &  &  \\
A_2 & A_3 & \cdots & A_{\ell} &  \\
A_1 & A_2  & \cdots & A_{\ell-1} &  A_{\ell} 
\end{array} \right) \longrightarrow \left( \begin{array}{ccccc} 
D(A_{0}) & & &  & 0 \\
D(A_{1}) & D(A_{0}) &  & &   \\
\vdots & \vdots & \ddots &  &  \\
D(A_{\ell-2}) & D(A_{\ell-3}) & \cdots & D(A_{0}) &  \\
D(A_{\ell-1}) & D(A_{\ell-2})  & \cdots & D(A_{1}) & D(A_{0}) 
\end{array} \right)=D\Gamma
\]
of $K$-vector spaces by
\[
\widetilde{\varphi}((x_{ij})) := (\varphi(x_{ij}))
\]
for any $(x_{ij}) \in M$. 
One can check that $\widetilde{\varphi}$ is an isomorphism of $\Gamma^{\op} \otimes_K \Gamma$-modules.

By the above description of $M$, if $A_0$ has finite global dimension, then we have a triangle-equivalence 
\[
\mathsf{D}(\Gamma,D\Gamma[1]) \simeq \underline{\mod}^{\mathbb{Z}/\ell\mathbb{Z}}A.
\]
by applying Theorem \ref{hull_vs_graded}.
\end{ex}

The second example is an application of Example \ref{sym_hull_vs_graded} to trivial extensions (see Example \ref{equ_Happel}).

\begin{ex}
Let $\Lambda$ be an algebra, and $A$ the trivial extension of $\Lambda$.
We consider the positive grading on $A$ given in Example \ref{equ_Happel}.
Then $A$ is a positively graded symmetric algebra of  Gorenstein parameter $1$ (Remark \ref{Happel<Chen}). 
Thus if $\Lambda$ has finite global dimension, then there exists a triangle-equivalence
\[
\mathsf{D}(\Lambda,D\Lambda[1]) \simeq \underline{\mathrm{mod}}A.
\]
by Example \ref{sym_hull_vs_graded}. 
Thus the stable categories of trivial extensions are always triangle-equivalent to ``$(-1)$-cluster categories''.
\end{ex}

\bigskip

Finally we continue to discuss Example \ref{ex_concrete1}.

\begin{ex}
We keep the notations in Example \ref{ex_concrete1}. 
So $A=K[x]/(x^{n+1})$ is a positively graded symmetric algebra with $\deg x=1$.
Then $A$ has Gorenstein parameter $n$. 
By Corollary \ref{sym_hull_vs_graded}, there exists a triangle-equivalence
\begin{eqnarray}
\underline{\mod}^{\mathbb{Z}/n\mathbb{Z}}A \simeq \mathsf{D}(\Gamma,D\Gamma[1]). \label{ex_hull_equ}
\end{eqnarray}
where $\Gamma$ is an $n \times n$ upper triangular matrix algebra over $K$.

Now we consider the special case $n=2$ as in Example \ref{ex_concrete1}.
Since a complete set of indecomposable $\mathbb{Z}/2\mathbb{Z}$-graded $A$-modules is given by $\{ F_2(X^i(j)) \ | \ i=1,2, \ j=0,1\}$, 
we obtain the Auslander-Reiten quiver of $\underline{\mod}^{\mathbb{Z}/2\mathbb{Z}}A$ by identifying each vertex $X$ with $X(2)$ in the Auslander-Reiten quiver \eqref{AR1}.
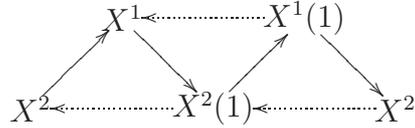
\[
\begin{xy}
(12,12) *{X^1}="F" ,
(0,0) *{X^2}="E" ,
(36,12) *{X^1(1)}="H" ,
(24,0) *{X^2(1)}="G" ,
(48,0) *{X^2}="I" ,

\ar "E" ; "F"
\ar "F" ; "G"
\ar "G" ; "H"
\ar "H" ; "I"
\ar@{.>} "G" ; "E"
\ar@{.>} "H" ; "F"
\ar@{.>} "I" ; "G"
\end{xy}
\]

On the other hand, let
\[
F:=- \Lten_{\Gamma}D\Gamma[1] :\mathsf{D}^{\mathrm{b}}(\mod \Gamma) \longrightarrow \mathsf{D}^{\mathrm{b}}(\mod \Gamma).
\]
Then we have $F \simeq \tau[2]$ by Theorem \ref{Serre=gldim}.
By this and since $\Gamma$ is hereditary, the orbit category $\mathsf{D}^{\mathrm{b}}(\mod\Gamma)/F$ has the natural structure of triangulated category by \cite[Section 4, Theorem]{Ke2}.
Thus we have $\mathsf{D}(\Gamma,D\Gamma[1]) = \mathsf{D}^{\mathrm{b}}(\mod\Gamma)/F$. 
We obtain the Auslander-Reiten quiver of $\mathsf{D}(\Gamma,D\Gamma[1])$ by identifying each vertex $X$ with $FX$ in the Auslander-Reiten quiver \eqref{AR2}.
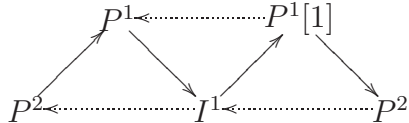
\[
\begin{xy}
(12,12) *{P^1}="F" ,
(0,0) *{P^2}="E" ,
(36,12) *{P^1[1]}="H" ,
(24,0) *{I^1}="G" ,
(48,0) *{P^2}="I" ,

\ar "E" ; "F"
\ar "F" ; "G"
\ar "G" ; "H"
\ar "H" ; "I"
\ar@{.>} "G" ; "E"
\ar@{.>} "H" ; "F"
\ar@{.>} "I" ; "G"
\end{xy}
\]

The above Auslander-Reiten quivers of $\underline{\mod}^{\mathbb{Z}/2\mathbb{Z}}A$ and $\mathsf{D}(\Gamma,D\Gamma[1])$ have the same shape, and this is a consequence of our triangle-equivalence Theorem \ref{hull_vs_graded}.

\end{ex}

\bigskip

\noindent \textit{Acknowledgement} 
This article is my Doctor thesis in Nagoya University.
I shall be deeply grateful to my supervisor O. Iyama for his contributions. 
He gave me so many helpful suggestions and discussions, and answered my questions patiently till I understood. 
I am also grateful to all staff in Nagoya University for taking care of me.
I would like to express his gratitude to B. Keller for answering to my questions about his paper \cite{Ke2} kindly.
I thank H. Asashiba, S. Oppermann and M. Herschend for their helpful comments.

\bigskip


\end{document}